\documentclass[a4paper,10pt]{article}
\usepackage[T1]{fontenc}
\usepackage[utf8]{inputenc}
\usepackage{amsmath,amsthm,amssymb}
\usepackage{verbatim}
\usepackage{graphicx}
\usepackage{bbold}
\usepackage[svgnames]{xcolor}
\usepackage{svg}
\usepackage{bbm}
\usepackage{calc}
\usepackage{caption}
\usepackage{enumitem}
\usepackage{mathtools}
\usepackage{color}
\usepackage{array}
\usepackage{adjustbox}
\usepackage{subfig}
\usepackage{pdflscape}
\usepackage{hyperref}
\hypersetup{
    colorlinks = true,
    linkbordercolor = {red},
   	linkcolor ={teal},
	anchorcolor = {pink},
	citecolor =  {orange},
	filecolor = {teal},
	menucolor = {teal},
	runcolor =  {teal},
	urlcolor = {teal},
}
\usepackage{multirow}
\usepackage{tabu}
\newlength\colOne % or some other suitable name
\newlength\colTwo

\newcommand\DivWidth[3]{\setlength{#1}%
   {(\widthof{#2}-\tabcolsep*(2*{#3}-2)-\arrayrulewidth*({#3}-1))/{#3}}}
\newcolumntype{C}[1]{>{\centering\let\newline\\\arraybackslash\hspace{0pt}}p{#1}}

\newlength{\rad}
\usepackage{authblk}
\usepackage{}
\usepackage{stmaryrd}
\usepackage{tikz}
\usepackage{tikz-cd}
\usepackage{quiver}
\usepackage{enumitem}
\setlist{noitemsep}

\theoremstyle{definition}
\newtheorem{thm}{Theorem}[section]
\newtheorem*{thm*}{Theorem}
\newtheorem{prop}[thm]{Proposition}
\newtheorem{lem}[thm]{Lemma}
\newtheorem{cor}[thm]{Corollary}
\newtheorem{defn}[thm]{Definition}
\newtheorem{remark}[thm]{Remark}

\newtheorem{exa}[thm]{Example}

\newcommand{\Blb}{\mathbb{B}}

\newcommand{\Nbb}{\mathbb{N}}

\newcommand{\Zbb}{\mathbb{Z}}

\newcommand{\Acal}{\mathcal{A}}

\newcommand{\Pcal}{\mathcal{P}}

\DeclareMathOperator{\al}{\alpha}
\DeclareMathOperator{\tr}{tr}
\DeclareMathOperator{\ellb}{lb}
\DeclareMathOperator{\lmp}{lmp}
\DeclareMathOperator{\rb}{rb}
\DeclareMathOperator{\rmp}{rmp}

\DeclareMathOperator{\supp}{supp}

\title{From free idempotent monoids to free multiplicatively idempotent rigs}

\author{Morgan Rogers\footnote{email: rogers@lipn.univ-paris13.fr}}
\affil{\small{LIPN -- UMR 7030 CNRS, Universit\'{e} Sorbonne Paris Nord}}

\date{\small{}}

\begin{document}

\maketitle

\begin{abstract}
A multiplicatively idempotent rig (which we abbreviate to \emph{mirig}) is a rig satisfying the equation $r^2 = r$. We show that a free mirig on finitely many generators is finite and compute its size. This work was originally motivated by a collaborative effort on the decentralized social network Mastodon to compute the size of the free mirig on two generators.
\end{abstract}

\section*{Introduction}

The algebraic and combinatorial problems addressed in this paper arose in the context of an online discussion. As such, this paper serves in part to illustrate the role of online interactions in the practice of mathematics: questions facilitated by the low stakes of such exchanges can lead to informal collaborations, unearthing of obscure references and curiosity-driven research. We hope that the present paper, where we make an effort to credit those involved, will encourage people to engage in online discussions about mathematics, since the relatively low barrier of entry can serve as an on-ramp for the mathematically curious.

On the other hand, our record of the discussion which motivated the writing of this paper shows the volatility of such interactions as a record of communal knowledge: while evidence of these interactions may persist for the medium term\footnote{We shall not speculate in detail on the life expectancy of the platforms where these discussions took place, but the history of the internet leads us to expect a span of years or decades rather than centuries.}, they have been buried in the intervening time to the extent that a casual reader would have practically no chance of coming across the discussions we reference here. A conclusion of the present work, therefore, is that academics leading or participating in online discussions (specifically in mathematics or domains where knowledge can be gained without recourse to controlled experiments) should also engage in the practice of formalizing the outcomes of those discussions into a form more widely accessible, while ensuring adequate attributions to and consent from the participants. The format of open-access publication seems particularly well-adapted to this aim. We should be clear that this paper goes well beyond the content of the discussions that inspired it, driven by the author's own curiosity.

From a mathematical perspective, the subject of this paper is the study of \textit{multiplicatively idempotent rigs} (mirigs, Definition \ref{def:mirig}) and more specifically \textit{free} mirigs. Perhaps surprisingly, these are finite (Lemma \ref{lem:Rnfinite}). Accordingly, the main problem that we address in this paper is `\textbf{How many elements does the free mirig on $n$ generators have?}' The original discussions went as far as establishing the size of the free mirig on two generators ($284$ elements, Example \ref{exa:R2count}) via computational methods which would be difficult to scale up. Here we get as far as computing the size of the free mirig on three generators, developing enough understanding of these structures to be able to effectively carry out the computation by hand. To summarize, the sequence computed thus far is:
\[|R_n| = 4,13,284,510605,\dotsc\]

The structure of the paper is as follows.

In Section \ref{sec:Boole}, we present the question asked by Jean-Baptiste Vienney which began the discussion and give historical context for the concepts featuring in his question. In particular, we record the (potentially surprising) fact that mirigs need not be commutative and explain why this leads us to disassociate the name `Boole' from these structures.

In Section \ref{sec:idMon}, we summarise results from a paper of Green and Rees \cite{GreenRees} regarding \textit{idempotent monoids} (also called \textit{bands}), whose structure is essential to understanding mirigs, and deriving a representation of elements of idempotent monoids by a certain kind of binary tree (Section \ref{ssec:trees}).

Using the tree presentation, we examine the subsemigroups of free idempotent monoids in Section \ref{sec:subsemi}. More particularly, we aim to understand their \textit{replete} subsemigroups (Definition \ref{defn:replete}) because these will be used later as part of the data representing elements of free mirigs. In the process, we extract some further combinatorial objects (leftmost and rightmost paths, Definition \ref{defn:leftmost}) simplifying the presentation of replete subsemigroups.

In Section \ref{sec:mirig}, we establish some basic theory for mirigs, including their \textit{characteristic} (Lemma \ref{lem:char}) and provide an example promised earlier illustrating non-commutativity. From Section \ref{ssec:direct} onwards we focus in on free mirigs and show that these can be presented in terms of formal sums of trees (\textit{forests} and \textit{thickets}, Definition \ref{defn:thickets}), eventually deriving from these a representation of elements of free mirigs involving replete subsemigroups of free idempotent monoids (\textit{complementary triples}, Definition \ref{defn:complementary}). From here we proceed to the computation of the sizes of free mirigs. Our biggest result is that the free mirig on three generators has $510605$ elements; see Example \ref{exa:R3count}.

Finally, in Section \ref{sec:examples}, we go beyond free mirigs in order to record some examples of mirigs which featured in the original discussions.

\subsection*{Acknowledgements}

I would like to thank all of the participants in the discussions that inspired this article, who are variously named at relevant points. I would also like to thank the friends and colleagues who humoured me in musing about the curiously large numbers that arose in this investigation.

\section{Rigs (feat. George Boole)}
\label{sec:Boole}

On December 19th 2022, Jean-Baptiste Vienney asked the internet (both on the \href{https://mattecapu.github.io/ct-zulip-archive/stream/266967-general:-mathematics/topic/topic_Are.20boolean.20rigs.20commutative.3F.html#316746498}{Category Theory Zulip forum} and on the \href{https://nforum.ncatlab.org/discussion/15707/boolean-rig/#Item_1}{nForum}) whether a rig in which all elements are multiplicatively idempotent is commutative.

\begin{defn}
\label{def:rig}
Recall that a \textbf{rig} (sometimes called a \textit{semiring}) consists of a set $R$ containing distinguished elements $0$ and $1$ and equipped with binary operations ${+}$ and ${\cdot}$ (called \textit{addition} and \textit{multiplication}) such that for all elements $x,y,z \in R$ the following eight (technically eleven) axioms are satisfied:
\begin{align*}
(x+y)+z &= x+(y+z) & x+0 = x &= 0+x \\
x+y &= y+x & x \cdot 1 = x &= 1 \cdot x\\
(x\cdot y) \cdot z &= x \cdot (y \cdot z) & x \cdot 0 = 0 &= 0 \cdot x\\
x \cdot (y+z) &= (x\cdot y) + (x \cdot z) & & \\
(x+y) \cdot z &= (x\cdot z) + (y \cdot z). & &
\end{align*}
In the present article we shall drop the symbol ${\cdot}$, writing $xy$ for $x \cdot y$, and use the associativity axioms (the first and third in the left-hand list) to drop parentheses. We can also drop parentheses on the right-hand side of the distributive laws (fourth and fifth in the left-hand list) by adopting the convention that multiplication takes precedence over addition. We follow the usual convention of writing a number $n$ to represent the sum of $n$ copies of $1$, so $1+1 = 2$ and $x+x+x = 3x$. Similarly, we write $x^n$ for $n$ copies of $x$ combined by multiplication.
\end{defn}

The motivating example of a rig is that of the \textit{natural numbers}, where addition and multiplication take their usual interpretations. However, the multiplication operation in the rigs we will consider here will behave rather differently to the multiplication of natural numbers.

\begin{defn}
\label{def:mirig}
A \textbf{multiplicatively idempotent rig}, which we abbreviate to \textbf{mirig}, is a rig (Definition \ref{def:rig}) in which the additional equation $x^2 = x$ is satisfied for all elements $x$.
\end{defn}

Before turning to mirigs, which are the subject of this article, we should mention that Vienney used the term `Boolean rigs', which is still the name associated with this structure on the nLab at time of writing \cite{nLabBoolRig}. It's worth examining why he chose this name and why we depart from it.

Much more widely studied than the rigs of Definition \ref{def:rig} are \textbf{rings}, where an axiom is added ensuring that each element has an additive inverse: a ri\textbf{n}g is a rig with \textbf{n}egatives\footnote{These naming conventions historically arose in the reverse order, of course.}. Formally, we recover the definition of ring by adding to Definition \ref{def:rig} a unary operation $-$ subject to the condition $x+(-x) = 0$. We can deduce that $-x = (-1) \cdot x$ and $(-1)^2 = 1$ in any ring. %Note that subtraction makes the axiom $x \cdot 0 = 0 = 0 \cdot x$ redundant.

Amongst rings, those most relevant in classical logic and lattice theory are \textbf{Boolean rings}: rings in which $x^2 = x$ for every element $x$, such as the two-element ring $\Zbb/2\Zbb$. Any Boolean ring is commutative ($xy = yx$ for all $x,y$), which is demonstrated by first observing that $-1 = (-1)^2 = 1$, then that
\begin{equation}
\label{eq:BRcomm}
x + y = (x + y)^2 = x^2 + xy + yx + y^2 = x + xy + yx + y,	
\end{equation}
and finally $xy - yx = xy + yx = 0$ by subtracting $x+y$.

Boolean rings are named for their connection with \textit{Boolean algebras}, which were introduced by George Boole a in the 19th century, \cite{Boole}. For the sake of completeness, we include a definition of Boolean algebra here.

\begin{defn}
\label{def:BA}
A \textbf{Boolean algebra} is a rig $B$ whose distinguished elements $0$ and $1$ we rename to $\bot$ and $\top$, respectively, and whose binary operations ${+}$ and ${\cdot}$ are denoted ${\vee}$ and ${\wedge}$ and referred to as \textit{join} and \textit{meet}, respectively\footnote{Curiously, Boole used the notation we introduced above for addition and multiplication for the meet and join operations. We use the modern notation here to distinguish Boolean algebras from Boolean rings.}. $B$ is equipped with a unary \textit{negation} operation ${\neg}$ subject to the following axioms:
\begin{align*}
x \wedge x &= x &
x \wedge \neg x &= \bot \\
x \vee x &= x &
x \vee \neg x &= \top
\end{align*}
\end{defn}
The archetypal Boolean algebra is $\Blb = \{\top,\bot\}$, whose elements are typically interpreted as `truth values'. This is just one of many possible axiomatizations, in which idempotency of both operations is assumed. From these, we can deduce consequences including: double negation elimination ($x = \neg\neg x$); uniqueness of negation (if $x \vee z = \top$ and $x \wedge z = \bot$, then $z = \neg x$), and the De Morgan laws ($\neg(x \vee y) = \neg x \wedge \neg y$ and dually $\neg(x \wedge y) = \neg x \vee \neg y$). The last of these enables us to deduce commutativity of meet from commutativity of join. It is also possible to define an ordering from the meet and join operations making Boolean algebras a special case of \textit{distributive lattices}, although we shall not explore that in depth here.
% Double negation:
% \begin{align*}
% x &= x \vee (\neg x \wedge \neg\neg x) = (x \vee \neg x) \wedge (x \vee \neg \neg x) \\
% &= \top \wedge (x \vee \neg \neg x) = (x \vee \neg \neg x) \wedge \top \\
% &= (x \vee \neg\neg x) \wedge (\neg x \vee \neg\neg x) = (x \wedge \neg x) \vee \neg\neg x = \neg \neg x
% \end{align*}
% Uniqueness of negation: if $x \vee z = \top$ and $x \wedge z = \bot$, then $z = \neg x$.
% \begin{align*}
% \neg x &= (x \wedge z) \vee \neg x = (x \vee \neg x) \wedge (z \vee \neg x) \\
% &= \top \wedge (z \vee \neg x) = (z \vee \neg x) \wedge \top \\
% &= (z \vee \neg x) \wedge (z \vee x) = z \vee (\neg x \wedge x) = z
% \end{align*}
% Negation of constants:
% \begin{align*}
% \bot \vee \top &= \top & \bot \wedge \top &= \bot
% \end{align*}
% whence $\neg 0 = 1$ and $\neg 1 = 0$.
% Absorption:
% \[1 \vee x = \neg x \vee x \vee x = \neg x \vee x = 1 \]

Any Boolean ring gives rise to a Boolean algebra with the same elements (up to renaming) by defining the negation, join and meet operations respectively as $\neg x := 1+x$, $x \vee y := x+y+xy$ and $x \wedge y = xy$. Conversely, any Boolean algebra yields a Boolean ring by taking addition to be \textit{symmetric difference}, $x+y := (x \wedge \neg y) \vee (\neg x \wedge y)$ and multiplication to be meet. The Boolean ring $\Zbb/2\Zbb$ gives rise to the Boolean algebra $\Blb$ under this correspondence.

\begin{remark}\label{rem:idemmon1}
Observe that if we keep only the meet operation of Boolean algebras, we arrive at a correspondence between meet semilattices and commutative idempotent monoids; we shall examine idempotent monoids in more detail in Section \ref{sec:idMon}.
\end{remark}

%Any finitely generated Boolean ring is a product of copies of the ring $\Zbb/2\Zbb$. Indeed, this ring is a \textit{field} and these are precisely the finite-dimensional algebras over this ring. 

To generalize Boolean rings, several incompatible definitions of `Boolean semiring' have been introduced in the past sixty years, as discussed by Guzm\'{a}n in \cite{BSring}. Guzm\'{a}n's own proposed definition comes from considering the variety of rigs generated by the ring $\Zbb/2\Zbb$ and the rig $\Blb$, which in particular includes both Boolean rings and Boolean algebras viewed as rigs (ignoring negation and subtraction, respectively, although they can be recovered when they exists).

Guzm\'{a}n characterizes this variety of rigs as those in which all elements $x$ satisfy the equations $1+x+x = 1$ and $x^2 = x$. These rigs are commutative by the following reasoning, which is rather more involved than \eqref{eq:BRcomm}.
\begin{equation}
\label{eq:xy=xyx}
\begin{split}
	xy &= xy(1+x+x) = x(y+yx+yx) \\%\nolabel\\
	&= x((y+yx)^2+yx) = x(y+yx+yx+yxy+yx) \\%\nolabel\\
	&=x(y+yxy+yx)=xy+xyxy+xyx \\%\nolabel\\
	&=xyx(y+y+1)=xyx. %
\end{split}
\end{equation}
Performing the same sequence of substitutions on the left instead, we have $yx = xyx$ too, so we conclude that $xy=yx$.

We can define the meet, join and negation operations for Guzm\'{a}n's Boolean semirings as we did for Boolean rings to obtain a distributive lattice structure satisfying $x \vee \neg x = 1$, but where the interaction between meet and negation is weakened to $x \wedge \neg x = x \wedge \neg 1$. Note that the join thus defined is still idempotent since:
\begin{equation}
\label{eq:Guzmanjoin}
x \vee x = x+x+x^2 = x+x^2+x^2 = x(1+x+x) = x.
\end{equation}
Guzm\'{a}n calls the resulting structures \textit{partially complemented distributive lattices}, but they are still essentially Boolean in flavour, in the sense that $0$ and $1$ play a crucial role as the top and bottom elements (with $1+1 = \neg 1$ lying between these two values), and one could conceivably treat such a lattice as a collection of `truth values' for a suitably constructed logic.

\begin{remark}
\label{rmk:termclash}
Beware that other authors have claimed the term `Boolean semiring' for distinct notions. For instance, Guterman decided in \cite{Guterman} to define a Boolean semiring as a subsemiring of a powerset algebra (a Boolean algebra), a definition inconsistent with the usual definition of Boolean rings given above.
\end{remark}

All of the structures discussed so far have been shown to acquire commutativity indirectly from idempotency in some form. Vienney's question is thus a very natural one to ask. Comparing Guzm\'{a}n's definition of Boolean semiring with our naive definition of mirig in Definition \ref{def:mirig}, we see that it can be rephrased as: ``Is the axiom $1+x+x=1$ necessary to show that multiplication is commutative?''

The present author observed that there exists a non-commutative mirig if and only if the free mirig on two generators is non-commutative, so that considering the free mirig on two generators would resolve Vienney's question. Before that rig could be investigated (see Section \ref{ssec:direct}), a less daunting counterexample was provided by a construction of Tim Campion; see Example \ref{exa:idmon2mirig}. This counterexample relied on the existence of non-commutative idempotent monoids, the simplest to describe of which is the free idempotent monoid on two generators; see Example \ref{exa:M2}.

With these examples, it became clear that the link with logic that justified the (already increasingly tenuous) association of the earlier structures with Boole is broken. The derived meet operation fails to be commutative as soon as the multiplication operation does, while the join operation fails to be idempotent in free mirigs, even with zero generators (see Example \ref{exa:R0}, where $1 \vee 1 = 1+1+1 \neq 1$). Thus there is no longer a canonical ordering on mirigs that might endow their elements with a logical interpretation, and as such we have chosen not to associate the adjective `Boolean' with our objects of study.

\begin{remark}
\label{rem:terminology}
Before moving on we should make a final point about our choice of terminology. Graham Manuell pointed out in the Zulip discussion that the term \textit{idempotent semiring} is already used to refer to a rig in which the \textit{addition} is idempotent, so $x+x = x$ for all $x$. This includes the \textit{tropical semiring} of natural numbers where addition is the $\max$ operation and multiplication is the usual addition of natural numbers, which can cause some confusion at first sight.

It is for this reason that we selected the terminology `mirig', which is close to a name suggested by John Baez. We shall examine the intersection of the class of additively idempotent rigs with the class of mirigs in Section \ref{sec:examples}.
\end{remark}

\section{Idempotent monoids}
\label{sec:idMon}

To begin, let us state the definition of our preliminary objects of interest.

\begin{defn}
\label{def:idmon}
An \textbf{idempotent monoid}\footnote{Beware that these are commonly referred to as \textbf{bands} in semigroup theory literature, although we shall not employ this term.} is a set $M$ endowed with a unit element $1 \in M$ and a binary multiplication operation $(x,y) \mapsto xy$ such that $xx = x$ for every element $x$.
\end{defn}

The free (not necessarily commutative) rig on a given set of generators $A$ can be constructed in two steps as the free commutative monoid on the free monoid on $A$. Correspondingly, we can construct the free \textit{mirig} on $A$ as a quotient of the free commutative monoid on the free \textit{idempotent} monoid on $A$. A thorough understanding of free idempotent monoids will be indispensible for studying free mirigs, so we spend this section establishing all of the relevant theory.
% \[\begin{tikzcd}[ampersand replacement = \&]
% \Set \&\& \Mon \&\& \rading.?
% \arrow[""{name=0, anchor=center, inner sep=0}, curve={height=-6pt}, from=1-1, to=1-3]
% \arrow[""{name=1, anchor=center, inner sep=0}, curve={height=-6pt}, from=1-3, to=1-5]
% \arrow[""{name=2, anchor=center, inner sep=0}, curve={height=-6pt}, from=1-5, to=1-3]
% \arrow[""{name=3, anchor=center, inner sep=0}, curve={height=-6pt}, from=1-3, to=1-1]
% \arrow["\bot"{description}, draw=none, from=0, to=3]
% \arrow["\bot"{description}, draw=none, from=1, to=2]
% \end{tikzcd}\]
% 
% \[\begin{tikzcd}[ampersand replacement = \&]
% \Set \&\& \IdMon \&\& \IdRig.
% \arrow[""{name=0, anchor=center, inner sep=0}, curve={height=-6pt}, from=1-1, to=1-3]
% \arrow[""{name=1, anchor=center, inner sep=0}, curve={height=-6pt}, from=1-3, to=1-5]
% \arrow[""{name=2, anchor=center, inner sep=0}, curve={height=-6pt}, from=1-5, to=1-3]
% \arrow[""{name=3, anchor=center, inner sep=0}, curve={height=-6pt}, from=1-3, to=1-1]
% \arrow["\bot"{description}, draw=none, from=0, to=3]
% \arrow["\bot"{description}, draw=none, from=1, to=2]
% \end{tikzcd}\]
% In order to show that the free idempotent rig on finitely many generators is finite, we start by recalling that this is true of free idempotent monoids, a result originally proved by Green and Rees.

We denote by $[n]$ an $n$-element set; where a specific set is required for generic $n$, this can be taken to be $\{0,1,\dotsc,n-1\}$, but for small $n$ we shall typically take this to be the set consisting of the first $n$ letters of the roman alphabet $\{a,b,c,\dotsc\}$. Let $F_n$ denote the free monoid on $[n]$, whose elements can be presented as words formed from the generators. Let ${\sim}$ be the congruence on $F_n$ generated by $w \sim ww$ for all words $w$. The free idempotent monoid on $n$ generators is $M_n := F_n /{\sim}$.

\begin{exa}
\label{exa:M2}
The free idempotent monoid on zero generators is the trivial monoid, while for one generator we have the underlying multiplicative monoid of $\Zbb/2\Zbb$. The free idempotent monoid on two generators $\{a,b\}$ can be represented by the set of square-free words on the generators, namely, $\{\epsilon,a,b,ab,ba,aba,bab\}$. One can prove that any word in $\{a,b\}^*$ is ${\sim}$-equivalent to a unique one of these square-free words. Note that since $ab \not\sim ba$, $M_2$ is not commutative.
\end{exa}

\begin{thm}[{\cite[Theorem 1]{GreenRees}}]
\label{thm:idmonfin}
$M_n$ is finite for all $n$, with cardinality
\begin{equation}
\label{eq:sum}
 	\sum_{k=0}^n \binom{n}{k} \prod_{i=1}^k (k-i+1)^{2^i}.
 \end{equation}
\end{thm}

This result might be surprising at first sight, since based on Example \ref{exa:M2} one might expect the set of elements of a free idempotent monoid to coincide with the set of square-free words in the generators, and \textit{there are infinitely many square-free words} on three or more generators! That last fact was originally discovered by Thue, \cite{Thue} (that's 1906, not 2006!); see \cite{ThueTranslate} for a more accessible, translated version of their papers. However, it turns out that the reduction of a word to a square-free word is not confluent when there are more than two generators.
\begin{exa}
Dylan McDermot gave the following small example illustrating non-confluence of reductions in the Zulip discussion:
\begin{align*}
a\textbf{babcbabc} & \to \textbf{abab}c \to abc \\
\textbf{abab}cbabc & \to abcbabc
\end{align*}
These reductions witness the fact that $abc \sim abcbabc$, so these words represent the same element of the free idempotent monoid on generators $a,b,c$, even though both are square-free. An older example can be found in \cite[\S2.4]{squareFree}. 
\end{exa}
This failure of confluence is significant enough to make ${\sim}$-equivalence classes so large that there are only finitely many of them. 

\subsection{Green-Rees Forms (GRFs)}

Theorem \ref{thm:idmonfin} is a result of Green and Rees, \cite{GreenRees}; this result and its consequences have been rediscovered numerous times since. In the present subsection we reproduce a proof of their result (structurally similar to the original proof) based on the version which can be found in \cite[\S2.4]{squareFree}. We shall require the structures developed in this proof to deepen our understanding of free idempotent monoids.

For a word $w \in F_n$, we define the \textbf{alphabet} of $w$ to be the set $\al(w) \subseteq [n]$ of generators appearing in $w$ and denote the \textbf{length} of $w$ by $|w|$.

\begin{lem}
\label{lem:decompose}
The alphabet descends to a well-defined function $\alpha: M_n \to \Pcal([n])$. In particular, $M_n$ decomposes as a disjoint union of subsemigroups,
\[M_n = \coprod_{A \subseteq [n]} \{w \in M_n \mid \alpha(w) = A\}.\]
\end{lem}
\begin{proof}
By inspection of the congruence ${\sim}$, if $w \sim w'$ then $\alpha(w) = \alpha(w')$. The decomposition follows by observing that $\alpha(ww') = \alpha(w) \cup \alpha(w')$, whence these are indeed disjoint subsemigroups.
\end{proof}

\begin{remark}
\label{rem:anticomm}
This is an instance of Clifford's decomposition of a semigroup over an idempotent semigroup \cite{Clifford}; in this case, the indexing idempotent semigroup is the powerset $\Pcal([n])$ with the union operation, and the fiber semigroups have the property of being \textbf{anticommutative}: if $uv = vu$ then $u=v$. This will be clearer when we present the elements of $M_n$ in a different form in Section \ref{ssec:trees} below; note that anticommutativity does \textit{not} hold for $M_n$ as a whole when $n>1$, since $a(aba) = aba = (aba)a$. By a result of McLean \cite[Lemma 1]{McLean}, this anticommutativity property holds for an idempotent semigroup if and only if $xyz = xz$ for all elements $x,y,z$.
\end{remark}

To better understand $M_n$, we must do further work at the level of $F_n$.

\begin{lem}
\label{lem:suffix}
Let $x,y \in F_n$ and suppose $\al(x) \supseteq \al(y)$. Then there exists $u \in F_n$ such that $x \sim xyu$.
\end{lem}
\begin{proof}
By induction on $|y|$. If $y$ is the empty word, we can let $u$ be empty also. Now suppose $y = y'a$ with $a$ a generator and that the theorem is true for words up to length $|y'|$. By assumption, $x \sim xy'u'$ for some $u'$. Moreover, $a \in \al(x)$ so we may write $x = x''ax'$. Let $u = x'y'u'$. Then we have:
\begin{equation*}
	xyu = x''ax'y'ax'y'u' \sim x''ax'y'u' = xy'u' \sim x. \qedhere
\end{equation*}
\end{proof}

The \textbf{Green-Rees form} (GRF) of a non-empty word $w$ is the tuple $(p,a,b,q)$ where:
\begin{itemize}
	\item $p$ is the maximal prefix of $w$ containing all but one of the generators in $\al(w)$,
	\item $a$ is the sole member of $\al(w) - \al(p)$, so $w = paw'$ for some word $w'$.
	\item Dually, $q$ is the maximal suffix of $\al(w)$ containing all but one of the members of $\al(w)$,
	\item $b$ is the sole member $\al(w)-\al(q)$.
\end{itemize}

\begin{lem}
\label{lem:GRFsim}
Suppose $w$ has GRF $(p,a,b,q)$. Then $w \sim pabq$.
\end{lem}
\begin{proof}
Let $x = pabq$. By construction of the GRF, we have $w = paw'$ and $w = w''bq$ for words $w',w''$. Since $\al(w') \subseteq \al(w) = \al(pa)$, by Lemma \ref{lem:suffix} there exists $u$ with $wu = paw'u \sim pa$. Similarly, since $\al(pa) = \al(bq)$, by the dual of Lemma \ref{lem:suffix} there exists $v$ with $vx = vpabq \sim bq$. Thus:
\begin{align*}
	x = pabq &\sim wubq & \text{and} \\
	w = w'bq &\sim w'vx.
\end{align*}
Using these relations, we have:
\[w \sim w'vx \sim w'vxx \sim wx \sim wwubq \sim wubq \sim x. \qedhere\]
\end{proof}

Thus GRFs provide candidate `standard' representatives of elements of the free idempotent monoid. The crucial fact which will enable us to use these to count the elements of $M_n$ is the following.

\begin{lem}
\label{lem:GRFequiv}
Let $x,y$ have respective GRFs $(p,a,b,q)$ and $(p',a',b',q')$. Then $w \sim w'$ if and only if $a=a'$, $b=b'$, $p \sim p'$ and $q \sim q'$.
\end{lem}
\begin{proof}
The `if' direction follows from Lemma \ref{lem:GRFsim}. Conversely, suppose $x \sim y$; given how the congruence ${\sim}$ is generated, it suffices to check the case that $x = rst$ and $y = rs^2t$ for words $r,s,t$. We have $x = pax'$ for some word $x'$. Consider two cases:
\begin{itemize}
	\item If $|rs|>|p|$, then we have $rs = pau$ for some word $u$. Thus $y = paust$ and $\al(p) = \al(y) - \{a'\}$, so by definition of the GRF, $p'=p$ and $a' = a$.
	\item If $|rs| \leq |p|$, there is a word $v$ such that $p = rsv$ and $t = vat'$. Thus $y = rs^2vat'$ and $\al(rs^2v) = \al(rsv) = \al(y) -\{a\}$, so $p' = rs^2v \sim rsv = p$ and $a' = a$, as required.
\end{itemize}
Applying the argument dually for $b,b',q,q'$ completes the result.
\end{proof}

\begin{remark}
\label{rem:Gerhard}
The work of Green and Rees was used by Gerhard (and others) to complete a classification of the equational subvarieties of idempotent semigroups, \cite{Gerhard}. Surprisingly, each subvariety is determined by a single equation (in addition to the idempotency equation $x^2 = x$), in contrast with more general varieties of semigroups. See the survey of varieties of semigroups by Evans \cite{Evans} for more information on this topic.
\end{remark}

\subsection{From Green and Rees to trees}
\label{ssec:trees}

We already have what we need to complete the proof of Theorem \ref{thm:idmonfin} by inductively counting the number of GRFs, as Green and Rees did. However, for understanding mirigs, it will serve us to develop GRFs into an explicit combinatorial representation of $M_n$ as a monoid of \textit{labelled, rooted binary trees}.

Formally, we define the set of trees $T_n$ alongside the function $\al:T_n \to \Pcal([n])$ returning a tree's \textbf{alphabet} (where $\Pcal([n])$ is the powerset of $[n]$) inductively as follows: %and $h:T_A \to \Nbb$ as follows:
\[\frac{}{() \in T_n ; \al(()) = \emptyset} \] %; h(()) = 0}\]
\[\frac{t_0,t_1 \in T_n, \, a_0 \in [n] - \al(t_0), \, a_1 \in [n] - \al(t_1), \, \al(t_0) \cup \{a_0\} = \{a_1\} \cup \al(t_1)}{(t_0,a_0,a_1,t_1) \in T_n ; \al((t_0,a_0,a_1,t_1)) = \al(t_0) \cup \{a_0\} } \] %; h((s,a,b,t)) = h(s)+1}\]

We define the \textbf{height} of a tree $t$ to be $h(t) := |\al(t)|$. We write $(a)$ as shorthand for $((),a,a,())$ (noting that the two generators in a tree of height $1$ are necessarily equal). 

We can define a map $\tr: F_n \to T_n$ by induction on $|\al(w)|$: it sends the empty word to the trivial tree $()$ and sends a non-empty word $w$ with GRF $(p,a,b,q)$ to $\tr(w) := (\tr(p),a,b,\tr(q))$.

Given a (fully expanded) tree $t$, we can read off the generators in order of appearance to produce a word $W(t) \in F_n$ such that, by inductive application of Lemma \ref{lem:GRFequiv}, $w \sim W(t)$ if and only if $\tr(w) = t$. Beware that while we shall shortly endow $T_n$ with a monoid structure, the section $W: T_n \to F_n$ of $\tr$ thus constructed is not a monoid homomorphism.

\begin{exa}
It can be useful to display these formal syntactic trees as binary trees with root $\ast$, where the successive generators in the tree are nodes arranged on the appropriate sides. We may omit leaves coming from the trivial tree $()$. For example, $\tr(bcac) = (((b),c,b,(c)),a,b,((c),a,a,(c)))$ can be displayed as follows:
% https://q.uiver.app/?q=WzAsMTUsWzcsMywiXFxhc3QiXSxbNiwyLCJhIl0sWzgsMiwiYiJdLFsyLDEsImMiXSxbMywxLCJiIl0sWzExLDEsImEiXSxbMTIsMSwiYSJdLFswLDAsImIiXSxbMSwwLCJiIl0sWzQsMCwiYyJdLFs1LDAsImMiXSxbOSwwLCJjIl0sWzEwLDAsImMiXSxbMTMsMCwiYyJdLFsxNCwwLCJjIl0sWzAsMV0sWzAsMl0sWzEsM10sWzEsNF0sWzIsNV0sWzIsNl0sWzMsN10sWzMsOF0sWzQsOV0sWzQsMTBdLFs1LDExXSxbNSwxMl0sWzYsMTRdLFs2LDEzXV0=
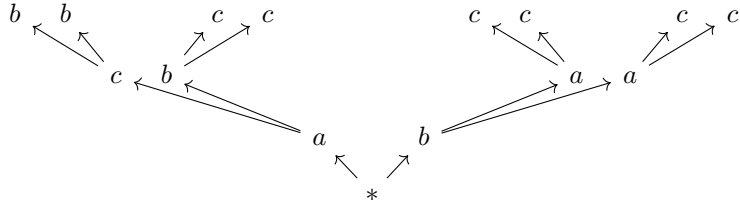
\[\begin{tikzcd}[column sep = tiny, row sep = small, ampersand replacement=\&]
	b \& b \&\&\& c \& c \&\&\&\& c \& c \&\&\& c \& c \\
	\&\& c \& b \&\&\&\&\&\&\&\& a \& a \\
	\&\&\&\&\&\& a \&\& b \\
	\&\&\&\&\&\&\& \ast
	\arrow[from=4-8, to=3-7]
	\arrow[from=4-8, to=3-9]
	\arrow[from=3-7, to=2-3]
	\arrow[from=3-7, to=2-4]
	\arrow[from=3-9, to=2-12]
	\arrow[from=3-9, to=2-13]
	\arrow[from=2-3, to=1-1]
	\arrow[from=2-3, to=1-2]
	\arrow[from=2-4, to=1-5]
	\arrow[from=2-4, to=1-6]
	\arrow[from=2-12, to=1-10]
	\arrow[from=2-12, to=1-11]
	\arrow[from=2-13, to=1-15]
	\arrow[from=2-13, to=1-14]
\end{tikzcd}\]
In either format we can read off that $W(\tr(bcac)) = bbcbccabccaacc$ (note the repetition coming from expansion of the abbreviated height $1$ trees), and we can independently verify that $bbcbccabccaacc \sim bcac$ by straightforward reductions.
\end{exa}

The $\tr$ map sends a word $w$ containing $n$ generators to a tree with $2^{n+1}-2$ labeled nodes in which every maximal path contains all $n$ generators. The above example illustrates that this presentation is very inefficient for short words, but it is uniquely defined thanks to the uniqueness of GRFs up to $\sim$. Moreover, the property of $W$ described above demonstrates that we if we endow $T_n$ with the monoid structure induced by $\tr$, we obtain a monoid isomorphic to $M_n$.

Before examining the monoid structure on $T_n$ in more detail, let us at last compute $|M_n| = |T_n|$ by induction on heights.

\begin{proof}[Proof of {Theorem \ref{thm:idmonfin}}]
Let $c_k$ be the number of trees of height $k$ on a given set $[k]$. Clearly $c_0 = 1$. For each $k \geq 1$, a tree of height $k$ labelled by generators in $[k]$ is constructed by choosing a pair of generators $a,b \in [k]$ and constructing a pair of trees of height $k-1$ on generators $[k]-\{a\}$ and $[k]-\{b\}$, respectively. Thus $c_k = k^2 c_{k-1}^2$, which inductively yields,
\begin{equation}
\label{eq:ck}
c_k = \prod_{i=1}^k (k-i+1)^{2^i}.
\end{equation}
To obtain the identity \eqref{eq:sum} of Theorem \ref{thm:idmonfin}, we multiply the counts $c_k$ by the number of subsets of $[n]$ of size $k$, then sum over the possible values of $k$. This concludes the proof.
\end{proof}

\subsection{Products of trees}

In this section, for trees $s,t$, we write $s \ast t$ for $\tr(W(s)W(t))$; this is the monoid operation on $T_n$. \textit{Beware that we shall drop the ${\ast}$ symbol in subsequent sections for legibility.} To facilitate the examination of this operation, we extend the notation for a tree $t = (t_0,a^t_0,a^t_1,t_1)$ to the component trees via indices, so $t_0 = (t_{00},a^t_{00},a^t_{01},t_{01})$ and so on (observe that $a^{t_0}_1 = a^{t}_{01}$). %We write $\ellb_k(t)$ for the \textit{left branch} of the leftmost subtree of height $k$ of $t$ and $\rb_k(t)$ for the \textit{right branch} of the rightmost subtree of height $k$ of $t$. $\ellb_k(t)$ and $\rb_k(t)$ consist of the first (resp. last) $2^k - 1$ generators in the syntactic expression for $t$ (we leave the reader to formalize the count of parentheses if they so choose). For instance,
% \begin{align*}
% 	\ellb_2((((a),b,a,(b)),c,c,((a),b,a,(b)))) &= (a),b \\
% 	\rb_2((((a),b,a,(b)),c,c,((a),b,a,(b)))) &= a,(b).
% \end{align*}
% We write $\ellb(t)$ as shorthand for $\ellb_{h(t)}(t)$ and define $\rb(t)$ similarly. By construction, $t = (\ellb(t),\rb(t))$.

\begin{lem}
\label{lem:sastt}
Let $s,t \in T_n$. Then $s \ast t$ can be computed inductively as follows:
\begin{itemize}
	\item If $s=()$, then $s \ast t = t$,
	\item If $t=()$, then $s \ast t = s$ (consistent with the above when $s = t = ()$),
	\item Otherwise, %$s=(s_0,a,b,s_1)$ and $t= (t_0,a',b',t_1)$. From here,
	we construct $(s \ast t)_0$ and $a^{(s \ast t)}_0$ inductively:
	\begin{itemize}
		\item If $a^t_0 \notin \al(s)$, $(s \ast t)_0 = s \ast t_0$ and $a^{(s \ast t)}_0 = a^t_0$.
		\item Else, $(s \ast t)_0 = (s \ast t_0)_0$ and $a^{(s \ast t)}_0 = a^{(s \ast t_0)}_0$.
	\end{itemize}
	\item $(s \ast t)_1$ and $a^{(s \ast t)}_1$ are constructed dually.
\end{itemize}
\end{lem}
\begin{proof}
The base cases are immediate from the fact that the trivial tree represents the empty word. For the remainder, we iteratively calculate the GRF of the product to find the leftmost appearance of the rightmost instance of a generator: if $a^t_0 \notin \al(w)$, then $a^t_0$ must be this generator; otherwise the generator appears further to the left and we can restrict attention to a branch of smaller height.
\end{proof}

\begin{cor}
\label{cor:product}
Let $s,t \in T_n$. Then:
\begin{itemize}
	\item If $\al(s) = \al(t)$, we have $s \ast t = (s_0,a^{s}_0,a^{t}_1,t_1)$.
	\item More generally, $s$ is a left factor of $t$ if and only if $\al(s) \subseteq \al(t)$ and $s_0 = t_{0^{k}}$, where $k = 1+h(t)-h(s) > 0$ and $0^k$ denotes a string of $k$ zeros in the subscript.
\end{itemize}
\end{cor}
\begin{proof}
If $\al(s) \subseteq \al(t)$, when we compute the right-hand branch of $s \ast t$ we will continually skip the generators appearing in $s$ (we will always reach the `Otherwise...' branch in the algorithm of Lemma \ref{lem:sastt}) until $s$ is reduced to the empty tree, yielding $a^{s \ast t}_1 = a^t_1$ and $(s \ast t)_1 = t_1$. Applying this on both sides gives the first statement.

Observe that $s$ is a left factor of $t$, meaning $t=s \ast t'$ for some tree $t'$, if and only if $s \ast t = t$, by idempotency of $s$. This necessitates $\al(s) \subseteq \al(t)$, and conversely given $\al(s) \subseteq \al(t)$ we have $a^{s \ast t}_1 = a^t_1$ and $(s \ast t)_1 = t_1$ by the above. We have thus shown that $s$ is a left factor of $t$ if and only if $\al(s) \subseteq \al(t)$ and $(s \ast t)_0 = t_0$ and $a^{s \ast t}_0 = a^t_0$. It remains to show that the latter condition is equivalent to the hypothesis $s_0 = t_{0^{k}}$.

If $\al(s) = \al(t)$, we have $(s \ast t)_0 = s_0$, whence $(s \ast t)_0 = t_0$ if and only if $s_0 = t_0$.

Otherwise, $h(t) > h(s)$. Following the algorithm for $(s \ast t)_0$, we conclude that $s \ast t = t$ if and only if $a^t_0 \in \al(t) - \al(s)$ and $s \ast t_0 = t_0$. Indeed, if $a^t_0 \in \al(s)$ then $t_0$ contains a generator not lying in $\al(s)$ which cannot be $a^t_0$, whence $a^{s \ast t}_0 \neq a^t_0$. By reverse induction on the size of $\al(t) - \al(s)$, we eventually reduce to the case above, completing the proof.
\end{proof}

\begin{remark}
\label{rem:factors}
A consequence of Corollary \ref{cor:product} is that if $x,x'$ are both left factors of a tree $t$, then $x$ is a left factor of $x'$ or $x'$ is a left factor of $x$, although we should stress that both of these can be true even if $x \neq x'$.
\end{remark}

\section{Subsemigroups of free idempotent monoids}
\label{sec:subsemi}

As we shall see in Section \ref{ssec:forests}, we can present elements of the free mirig on $[n]$ as formal sums of trees in $T_n$: \textit{thickets}, Definition \ref{defn:thickets}. By ignoring how often each tree appears, a thicket determines a subset of $T_n$. Observing that $u+v \sim (u+v)^2 \sim u+v+uv+vu$ in a mirig, we can anticipate that a \textit{maximal} thicket representing a given element of the free mirig (supposing such exists) will determine a subset which is closed under products of trees: a \textit{subsemigroup} of $T_n$. For this reason, we pre-emptively study subsemigroups in the present section, beginning with subsemigroups of a particularly simple form, the \textit{uniform} subsemigroups, in Section \ref{ssec:uniform}. This will reveal that in identifying subsemigroups we can treat the left- and right-hand branches of the constituent trees almost independently, a fact which we exploit in the remainder of the article.

In fact, the subsemigroup determined by a maximal thicket must be closed under `internal multiplication', in the sense that we have $x(u+v)y \sim xuy + xvy+xuvy+xvuy$, so that if both $xuy$ and $xvy$ are (independently) present, then $xuvy$ and $xvuy$ must be too. This leads us to the notion of \textit{replete} subsemigroups in Section \ref{ssec:replete}. The parenthetical `independently' leads to some subtleties to be examined in later sections.

Since this paper is about counting structures, and since being able to count subsemigroups is a prerequisite for counting elements of the free mirig later, we enumerate the subsemigroups we find along the way. We count the uniform subsemigrops of $T_n$ in Theorem \ref{thm:uniformcount} and (with much greater effort!) count the replete subsemigroups of $T_n$ up to $n = 3$ (see Lemma \ref{lem:T2replete} and Proposition \ref{prop:T3subcount}).

\subsection{Uniform subsemigroups: from trees to branches}
\label{ssec:uniform}

We proved Theorem \ref{thm:idmonfin} using the partition of $T_n$ into subsemigroups using $\al$ inherited from Lemma \ref{lem:decompose}. In this section, we examine these subsemigroups more closely, since they can help us to understand more general subsemigroups. These developments will be pivotal to understanding mirigs later on.

\begin{defn}
We say a subsemigroup $S \subseteq T_n$ is \textbf{uniform} if $\al(t) = \al(t')$ for all $t,t' \in S$. Every subsemigroup of $T_n$ decomposes in a canonical way as a union of uniform subsemigroups. We write $\al(S)$ for the generators appearing in the trees of a uniform subsemigroup.

Given a uniform subsemigroup $S$ with $\al(S) \neq \emptyset$, define the set of \textbf{left branches} of $S$ to be $\ellb(S) := \{(t_0,a^{t}_0) \mid t \in S\}$; dually, we define the set of \textbf{right branches} $\rb(S)$ to be $\{(a^{t}_1,t_1) \mid t \in S\}$. We extend this with the convention that $\ellb(\{()\}) = \{()\} = \rb(\{()\})$.
\end{defn}

\begin{thm}
\label{thm:uniformcount}
Any inhabited uniform subsemigroup $S$ of $T_n$ is uniquely determined by the pair $\ellb(S),\rb(S)$, and conversely any inhabited choice of $\ellb(S),\rb(S)$ over a common alphabet determines a uniform subsemigroup of $T_n$. In total, the number of inhabited uniform subsemigroups of $T_n$ is therefore:
\begin{equation}
\label{eq:sum2}
\sum_{k=0}^n \binom{n}{k} \left(2^{\left(\prod_{i=0}^{k-1} (k-i)^{2^i}\right)}-1\right)^2.
\end{equation}
\end{thm}
\begin{proof}
Given a uniform subsemigroup $S$ and $t:= (t_0,a^t_0,a^t_1,t_1) \in \ellb(S) \times \rb(S)$, there must be trees $r:= (t_0,a^t_0,a^r_1,r_1)$ and $s:=(s_0,a^s_0,a^t_1,t_1)$ appearing in $S$ from which these branches arise, whence $rs=t \in S$ by Corollary \ref{cor:product}. Conversely, the only trees which can appear in $S$ are those whose left and right branches respectively appear in $\ellb(S)$ and $\rb(S)$. We can check manually that the claim extends to the exceptional case $S = \{()\}$.

For $k \geq 1$, the total number of left branches on a given set of $k$ generators is $kc_{k-1}$, where the constants $c_k$ are those computed in \eqref{eq:ck}. Conveniently, since $k = (k-0)^{2^0}$, we can present this as:
\begin{equation}
\label{eq:kck}
	kc_k = \prod_{i=0}^{k-1} (k-i)^{2^i}.
\end{equation}
The number of inhabited subsets of branches on this set of generators is $2^{kc_{k-1}}-1$, and similarly for the right branches. As such, we get $(2^{kc_{k-1}}-1)^2$ possible uniform subsemigroups on this set of generators. Moreover, when $k=0$, the product on the right-hand side of \eqref{eq:kck} evaluates to $1$ by the usual convention for empty products (even though the left-hand side of \eqref{eq:kck} evaluates to zero so the equation does not apply in this case), which yields the correct value of $1$ for the summand corresponding to $k = 0$. Summing over $k$ and the number of subsets of size $k$, we obtain the formula \eqref{eq:sum2}.
\end{proof}

\begin{defn}
\label{def:alphlonely}
Let $U \subseteq T_n$ now be an arbitrary subset. We define the \textbf{alphabets} of $U$ (note the plural) to be the set
\[\bar{\al}(U):=\{A \subseteq [n] \mid \exists t \in U, \, \al(t) = A\}.\]
Given $A \subseteq [n]$, we also define $U_{A} := \{t \in U \mid \al(t) = A\}$.

A \textbf{lonely tree} of $U \subseteq T_n$ is a tree $t \in U$ such that $U_{\al(t)} = \{t\}$. 
\end{defn}

For a subsemigroup $S$, $S_A$ is a uniform subsemigroup of $S$ for each $A$ and $\bar{\al}(S)$ is closed under unions in the powerset lattice $\Pcal([n])$. The minimal elements of $\bar{\al}(S)$ necessarily form an antichain in $\Pcal([n])$ (we never have $A \subseteq A'$ for $A \neq A'$ minimal, since this contradicts minimality of $A'$).

\begin{prop}
\label{prop:subsemibound}
Let $S \subseteq T_n$ be a subsemigroup and $A,B \in \bar{\al}(S)$. The following inequalities bound the sizes of branch sets of uniform subsemigroups of $S$:
\begin{enumerate}
	\item If $A \subseteq B$ then $|\ellb(S_{A})| \leq |\ellb(S_B)|$.
	\item If $A,B \in \bar{\al}(S)$ are \textit{incomparable} (we have $A \nsubseteq B$ and $B \nsubseteq A$) then $|\ellb(S_{A \cup B})| \geq |\ellb(S_A)| + |\ellb(S_B)|$.
	\item If moreover $A,B \in \bar{\al}(S)$ are \textit{disjoint} in $[n]$, then we can improve this to $|\ellb(S_{A \cup B})| \geq |S_A||\ellb(S_B)| + |\ellb(S_A)||S_B|$.
\end{enumerate}
Similar inequalities hold for right branches. In particular, $t \in S$ can only be lonely if $\al(t)$ and all its subsets in $\bar{\al}(S)$ are \textbf{union-irreducible}, in the sense that for $A_1,A_2 \in \bar{\al}(S)$, $A_1 \cup A_2 \subseteq \al(t)$ implies $A_1 \subseteq A_2$ or $A_2 \subseteq A_1$.
\end{prop}
\begin{proof}
For 1, let $N:=|\ellb(S_{A})|$ and let $s^1, \dotsc, s^N \in S_{A}$ be trees with distinct left branches (the superscript is intended to denote an index rather than a power here) and $t \in S_B$ (such exists by the assumption that $B \in \bar{\al}(S)$). Then $s^1t,\dotsc,s^Nt \in S_B$ (since $A \subseteq B$) and by the second part of Corollary \ref{cor:product}, the left branches of these trees are distinct, as required.

For 2, let $N$ and $s^1, \dotsc, s^N$ be as above, let $M := |\ellb(S_{B})|$ and let $t^1,\dotsc,t^M \in S_B$ with distinct left branches. We can similarly conclude that $s^1t^1,\dotsc,s^Nt^1$ and $t^1s^1,t^2s^1,\dotsc,t^Ms^1$ have distinct left branches: to see that branches from the left list are distinct from those in the right list, observe that $a^{t^is^1}_0$ must be a generator in $B \backslash A$ whereas $a^{s^jt^1}_0$ must be a generator in $A \backslash B$.

For 3, observe that each $s \in S_A$ appears as the leftmost subtree of $st$ of height $h(s)$ for each $t \in S_B$, but that distinct left branches of trees in $S_B$ produce distinct left branches of $st$. Thus we count a distinct left branch in $S_{A \cup B}$ for each tree in $S_A$ and left branch of $S_B$; the other summand is derived similarly from products in the reverse order, $ts$.
\end{proof}

Note that we cannot expect to do better in the second part of Proposition \ref{prop:subsemibound} in general: it can be the case that $s^1t^1$ and $s^1t^2$ have the same left branch, for instance, so we cannot improve the inequality without the extra hypothesis of disjointness appearing in the third part.

\begin{exa}
\label{exa:T2subsemi}
Let us manually count the subsemigroups of $T_2$. We first count the subsemigroups excluding the trivial tree, and organize the counting of subsemigroups of $S \subseteq T_2 \backslash \{1\}$ by their alphabets, $\bar{\al}(S)$. We can further simplify the task by accounting for permutations of the generators, as indicated by the `factor' row in Table \ref{table:T2count}.

We already know how to count the (empty or) uniform subsemigroups. If $a \in S$ and $S_{\{a,b\}}$ is inhabited, $S$ must contain $aba$, whence have at most $4$ choices of branch sets; if also $b \in S$ then $S_{\{a,b\}}$ must contain all four possible trees. Thus we arrive at the count in Table \ref{table:T2count}, with a total of $21$ possibilities. Doubling to count the subsemigroups including the trivial tree yields a total of $42$ subsemigroups.
\begin{table}[ht]
\centering
\begin{tabular}{ l |c|c|c|c|c|c|}
$\bar{\al}(S)$ & $\emptyset$ & $\{\{a\}\}$ & $\{\{a,b\}\}$ & $\{\{a\},\{a,b\}\}$ & $\{\{a\},\{b\},\{a,b\}\}$ \\
\hline
Factor     & $1$ & $2$ & $1$ & $2$ & $1$ \\
Count      & $1$ & $1$ & $9$ & $4$ & $1$
\end{tabular}
\caption{Summary of the number of subsemigroups of $T_2$ not containing the trivial tree on each possible alphabet (up to permutation of generators, accounted for by the factor).}
\label{table:T2count}
\end{table}
\end{exa}

%Considering the subsemigroups computed in Example \ref{lem:T2replete}, we find that the lower bounds provided by Proposition \ref{prop:subsemibound} are tight for $n=2$, but only because the `smaller' trees are simple enough not to inflate the count of left and right branches of `larger' trees.

%We leave further improvements and the calculation of corresponding lower bounds on the number of subsemigroups to the enthusiastic reader. Note that the above may well be the more feasible part of obtaining such bounds, coming after the selection of the union-subsemilattice $\bar{\al}(S) \subseteq \Pcal([n])$; see Remark \ref{rem:hardcount} for how a related problem makes our counting quest difficult later on.

%Let $S \subseteq T_n$ be a subsemigroup, $A \subseteq A' \in \bar{\al}(S)$ and $s \in S_A$. What can be said about $S_{A'}$? For each $t \in S_{A'}$, we must have $((s\ast t)_0,a^{s \ast t}_0) \in \ellb(S_{A'})$ and a dual constraint on $\rb(S_{A'})$. Moreover, these conditions are sufficient for $S_{A'}$ to be closed under multiplication by $s$ and they depend only on the left (resp. right) branch of $s$, since the opposite branch will be entirely determined by $s$. In other words, this is a condition that we can check branch-wise.

% Given $n$ generators, one in $n$ of the words of size $n$ have a given generator as a left factor..! Does that help with counting? (I don't know.)

\subsection{Replete subsemigroups: from branches to paths}
\label{ssec:replete}

Pre-empting the material in the sections to follow, we introduce a further class of subsemigroups of $T_n$ which are closed under some further operations. 

\begin{defn}
\label{defn:replete}
Let $U \subseteq T_n$ be any subset and $x,y \in T_n$ trees. We define the \textbf{$(x,y)$-factor of $U$} to be the subset
\[x\backslash U / y := \{t \in T_n \mid xty \in U\}.\]

Let $S \subseteq T_n$ be a subsemigroup. We say $S$ is \textbf{replete} if for all $x,y \in T_n$, the set $x \backslash S / y$ is a subsemigroup of $T_n$. 
\end{defn}

Observe that there is contravariant compositionality in the construction of $(x,y)$-factors, in the sense that for any $U \subseteq T_n$, $x' \backslash (x \backslash U / y) / y' = xx' \backslash U / y'y$. Moreover, $x\backslash U / y$ is empty unless $x,y$ are respectively left and right factors of some element of $U$. By inspection, $(x,y)$-factors respect subset operations, which we record in the following lemma.

\begin{lem}
\label{lem:xyfactors}
For a family of subsets $\{U_i \mid i \in I\}$ of $T_n$ and $x,y \in T_n$:
\begin{align*}
x \backslash \left( \bigcap_{i \in I} U_i \right) /y &=
\bigcap_{i \in I} \left( x \backslash U_i /y \right) \text{ and} \\
x \backslash \left( \bigcup_{i \in I} U_i \right) /y &=
\bigcup_{i \in I} \left( x \backslash U_i /y \right).
\end{align*}
\end{lem}

As a consequence, even before explicitly computing any structure of replete subsemigroups, we can conclude a result enabling us to abstractly produce replete subsemigroups of $T_n$ for general $n$.

\begin{cor}
\label{cor:repleteintersect}
Let $\{S_i \subseteq T_n \mid i \in I\}$ be a family of replete subsemigroups. Then $S:=\bigcap_{i \in I} S_i$ is replete. In particular, since the maximal subsemigroup $T_n$ is replete, any subset $U \subseteq T_n$ generates a replete subsemigroup, the intersection of all replete subsemigroups containing $U$, which we denote by $\langle U \rangle_r$.
\end{cor}

\begin{remark}
\label{rem:repletealphabets}
In fact, if $\Acal \subseteq \Pcal([n])$ is a subsemilattice, the maximal subsemigroup with $\Acal$ as its alphabet, $\{t \mid \alpha(t) \in \Acal\}$, is clearly replete, so for a given subsemigroup $S$, $\bar{\al}(\langle S \rangle_r) = \bar{\al}(S)$. In other words, generating a replete subsemigroup does not affect its alphabets.
\end{remark}
% \begin{proof}
% An immediate consequence of the fact that subsemigroups are closed under intersection and $(x,y)$-factors respect these (Lemma \ref{lem:xyfactors}).
% \end{proof}

%Some general observations about $(x,y)$-factors of a subsemigroup $S \subseteq T_n$. 

%Clearly $1 \in x \backslash U / y$ if and only if $xy \in U$, and $x \backslash U / y$ is always trivially closed under products with $1$.

% While $x\backslash f / y$ need not be a subsemigroup, it is always closed under the operation $(u,v) \mapsto uyxv$.

Repleteness is a closure under products of `internal factors' of trees, and it is this property which will make them useful in representing mirigs. It turns out that we have already encountered many replete subsemigroups.

\begin{lem}
\label{lem:T2replete}
All $42$ subsemigroups of $T_2$ from Example \ref{exa:T2subsemi} are replete. We can thus record also that $T_0$ has $2$ subsemigroups and $T_1$ has $4$.
\end{lem}
\begin{proof}
We work with words (in $M_2$) for conciseness. By the observations preceding Lemma \ref{lem:xyfactors} and symmetry it suffices to show that for any subsemigroup $S \subseteq M_2$, $a \backslash S / 1$ is also a subsemigroup for $a$ a generator. Indeed, suppose $u,v \in a \backslash S / 1$, which is to say that $au,av,auav \in S$; we wish to show that $auv \in S$. This is immediate if $au=a$ or $ua = u$ or $av=v$ or $uv=u$ or $uv=v$. By an exhaustive search, the only case in $M_2$ not satisfying one of these equations is $(u,v) = (ab,ba)$, in which case $auv = aabba = aababa = auav$ is in $S$, as required.
\end{proof}

As the reader may have expected, this result belies the general case.

\begin{exa}
\label{exa:T3notreplete}
In contrast with Lemma \ref{lem:T2replete}, $T_n$ has non-replete subsemigroups for $n \geq 3$. For generators $a,b,c$, consider the subsemigroup $S$ generated by the trees $\tr(ab)$ %((a),b,a,(b))$
and $\tr(ac)$, %$((a),c,a,(c))$,
which contains four more trees, namely 
$\tr(abac)$, %(((a),b,b,(a)),c,&b,((a),c,a,(c)))
$\tr(acab)$, %(((a),c,c,(a)),b,&c,((a),b,a,(b)))
$\tr(acabac)$ %(((a),c,c,(a)),b,&b,((a),c,a,(c)))
and $\tr(abacab)$. %(((a),b,b,(a)),c,&c,((a),b,a,(b)))
Then $(a) \backslash S /()$ contains $(b)$ and $(c)$ but not their product $((b),c,b,(c)) = \tr(bc)$, since by comparing trees we can easily check that $\tr(abc)$ is not in the list above.
\end{exa}

One might hope that singleton subsemigroups, at least, would be replete. For instance, we might observe that for any $x,t \in T_n$, $x \backslash \{t\} /()$ is a subsemigroup, since for $u,v \in x \backslash \{t\} /()$ we have $xu = t = xv$ whence $xuv = xvv =t$, ensuring $uv \in x \backslash \{t\} /()$ too. However, we will shortly see that this feature does not extend to other $(x,y)$-factors. Before reading on, the enthusiastic reader may enjoy the exercise of picking a tree of height $3$ and identifying how the singleton on that tree fails to be a replete subsemigroup.

\begin{lem}
\label{lem:uniformrepleteness}
A subsemigroup $S \subseteq T_n$ is replete if and only if $S_A$ is replete for each $A \in \al(S)$.
\end{lem}
\begin{proof}
If $S$ is replete, each $S_A$ automatically is too. Conversely, suppose that $S$ is a subsemigroup such that $S_A$ is replete for each $A \in \al(A)$, and suppose $xuy,xvy \in S$. Let $t = xuyxvy$ and $A:= \al(t)$. Observe that the right branch of $xuvy$ is the same as that of $txuvy$, and that the latter is in $S_A$ since $txuy,txvy \in S_A$ and $S_A$ is replete. A dual argument shows that the left branch of $xuvy$ appears in $S_A$; it follows that $xuvy \in S_A$ by Theorem \ref{thm:uniformcount}, whence $S$ is replete, as required.
\end{proof}

It follows from the proof of Lemma \ref{lem:uniformrepleteness} that repleteness of a subsemigroup $S \subseteq T_n$ depends only on its left and right branches. We will refer to a set of left or right branches as \textbf{replete} if it is a valid set of branches for a replete uniform subsemigroup. By Theorem \ref{thm:uniformcount}, replete uniform subsemigroups are in bijection with pairs consisting of a replete set of left branches and a replete set of left branches.

\begin{defn}
\label{defn:leftmost}
Given a tree $t$, define the \textbf{leftmost path} $\lmp(t)$ to be the sequence $(a^t_{0^{h(t)}},\dotsc,a^t_{00},a^t_0)$, and dually the \textbf{rightmost path} $\rmp(t)$ to be the sequence $(a^t_1,a^t_{11},\dotsc,a^t_{1^{h(t)}})$. Each is a total ordering on $\al(t)$.

We extend this terminology to left and right branches, respectively; for instance, given a right branch $(a,t)$, we define $\rmp(a,t) := (a,a^t_1,\dotsc,a^t_{1^{h(t)}})$.
\end{defn}

Since left and right branches behave symmetrically, we restrict our attention to rightmost paths in the remainder. To describe how rightmost paths interact with products of trees, we define the following operation: let
\begin{equation}
\label{eq:rmpast}
(a_1,\dotsc,a_k) \ast_r (a'_1,\dotsc,a'_{l}) := (a_{i_n},\dotsc,a_{i_1},a'_1,\dotsc,a'_l),
\end{equation}
where $n = |\{a_1,\dotsc,a_k\} \backslash \{a'_1,\dotsc,a'_l\}|$ and for each $1 \leq j \leq n$, $i_j = \max\{i \mid a_i \notin \{a_{i_{j-1}},\dotsc,a_{i_1},a'_1,\dotsc,a'_l\}\}$.

\begin{lem}
\label{lem:astseq}
Let $s,t \in T_n$. Then $\rmp(st) = \rmp(s) \ast_r \rmp(t)$; in particular, $\rmp(t)$ appears as a final subsequence of $\rmp(st)$. Also, $\rmp(tst) = \rmp(st)$.
\end{lem}
\begin{proof}
By inspection of the algorithm in Lemma \ref{lem:sastt} for computing $st$.
\end{proof}

\begin{prop}
\label{prop:repletepaths}
The minimal inhabited replete right branch sets are the sets of branches with a given rightmost path. That is, if $S \subseteq T_n$ is a uniform replete subsemigroup, $(a,t) \in \rb(S)$ and $\rmp(a',t') = \rmp(a,t)$ then $(a',t') \in \rb(S)$, and conversely the set $\{(a',t') \mid \rmp(a',t') = \rmp(a,t)\}$ is a replete right branch set. The dual is true for left branches.
\end{prop}
\begin{proof}
Suppose $S$ is a replete uniform subsemigroup. Let $s \in S$ and suppose $\rmp(t) = \rmp(s) = (a_1,\dotsc,a_k)$. Then $s = s(a_1)s_1 = st_1s_1$, whence $s(a_1)t_1s_1 \in S$ by repleteness, and this has right branch $(a_1,t_1s_1)$. Similarly, $t_1s_1 = t_1s_1(a_2)s_{11} = t_1s_1t_{11}s_{11}$, whence after multiplying on the left we conclude that $s(a_1)t_1s_1(a_2)t_{11}s_{11} \in S$, which has right branch $(a_1,(t_{10},a^t_{10},a_2,t_11s_11))$. Note that these computations rely on the fact that $a_1 \notin t_1$ and $a_2 \notin t_{11}$. Iterating the argument, we arrive at a tree in $S$ with the same right branch as $t$, as claimed.

Conversely, if $\rmp(xuy) = \rmp(xvy)$ then by the observations in Lemma \ref{lem:astseq} we have $\rmp(xuvy) = \rmp(xuy) \ast_r \rmp(vy) = \rmp(xvy)$, whence $\{(a',t') \mid \rmp(a',t') = \rmp(a,t)\}$ is a replete set of right branches, as claimed.
\end{proof}

From here on, we denote by %either
$[a_1,\dotsc,a_k \rangle$ %or $[(a_1,\dotsc,a_k) \rangle$ %(resp. $\langle a_k,\dotsc,a_1]$)
the set of right %(resp. left)
branches with rightmost path $(a_1,\dotsc,a_k)$. %(resp. leftmost path $(a_k,\dotsc,a_1)$).

\begin{lem}
\label{lem:repletebranchset}
The set of right branches $[a_1,\dotsc,a_k \rangle$ has cardinality
\[\prod_{j=1}^{k-1} j^{2^{k-j}-1}.\]
\end{lem}
\begin{proof}
This can be obtained by a direct counting argument, or by observing that all such branch sets must have the same cardinality and there are exactly $k!$ of them, so we can simply divide the earlier formula \eqref{eq:kck} by $k!$ after re-indexing via $j:= k-i$.
% An element of the set is obtained by choosing, for each $1 \leq i \leq k-1$, a generator in $\{a_{i+1},\dotsc,a_k\}$ and a tree of height $k-i-1$ on the remaining generators. This gives a factor of $(k-i)c_{k-i-1}$, and reindexing by $l = k-i$ produces the expression
% \[\prod_{l=1}^{k-1} \left(lc_{l-1}\right).\]
% Substituting the formula for $c_{l-1}$ from \eqref{eq:ck}, we obtain
% \[\prod_{l=1}^{k-1} l \prod_{i=1}^{l-1} (l-i)^{2^i}.\]
% Putting together powers of $j = (l-i)$ and using the identity $1+2+2^2+...+2^{m-1} = 2^m-1$ produces the claimed identity.
\end{proof}

\begin{thm}
\label{thm:repletechar}
A set $\rb(S)$ of right branches of height $k$ is replete if and only if it is a union of rightmost path equivalence classes and for $s,t \in S$ with $\rmp(s) = (a_1,\dotsc,a_k)$ and $\rmp(t) = (a'_1,\dotsc,a'_k)$ and each index $1 \leq j \leq k$ the branch equivalence class of the right-most path $(a_1,\dotsc,a_k) \ast_r (a'_j,\dotsc,a'_k)$ is contained in $\rb(S)$.
\end{thm}
\begin{proof}
To demonstrate necessity, let $x$, $u$ and $v$ be trees in $T_n$ such that $\rmp(x) = (a'_j,\dotsc,a'_k,a'_1,\dotsc,a'_{j-1})$, $\rmp(u) = (a_1,\dotsc,a_k)$ and $\rmp(v) = (a'_j,\dotsc,a'_k)$. % (we take $y = ()$).
The branch classes determined by $\rmp(xu),\rmp(xv)$ are contained in $\rb(S)$ by assumption, whence the class of $\rmp(xuv)$ is too; this is the required product.

For sufficiency, we again use the observation from the proof of Proposition \ref{prop:repletepaths} that $\rmp(xuvy) = \rmp(xuy) \ast_r \rmp(vy)$, whence the rightmost path of $xuvy$ is of the hypothesised form.
\end{proof}

Given an alphabet $[n]$, we may derive from \eqref{eq:rmpast} operations ${\ast}^j_r$ (with $1 \leq j \leq n$) on the set of sequences of distinct elements of $[n]$:
\begin{equation}
\label{eq:rmpastj}
(a_1,\dotsc,a_k) \ast^j_r (a'_1,\dotsc,a'_k) := 
\begin{cases}
(a_1,\dotsc,a_k) \ast_r (a'_j,\dotsc,a'_{l}) & 1 \leq j \leq l \\
(a_1,\dotsc,a_k) & \text{ otherwise.}
\end{cases}
\end{equation}
The condition given in Theorem \ref{thm:repletechar} can be presented in terms of the restriction of these operations to the sequences of length $k$. Beyond uniform branch sets, though, we can conclude that replete branch sets for $T_n$ correspond to subalgebras of the resulting structure, and consequently that replete subsemigroups correspond to pairs of such subalgebras having compatible alphabets in the obvious sense.

\begin{remark}
\label{rmk:automaticcases}
The case $j = 1$ in Theorem \ref{thm:repletechar} is automatically satisfied, since
\[(a_1,\dotsc,a_k) \ast^1_r (a'_1,\dotsc,a'_k ) =
(a'_1,\dotsc,a'_k).\]
More subtly, the case $j=2$ is also automatic since $a'_1$ will be the only generator appearing in $(a_1,\dotsc,a_k)$ but not $(a'_2,\dotsc,a'_k)$, ensuring
\[(a_1,\dotsc,a_k) \ast^2_r (a'_1,\dotsc,a'_k) =
(a'_1,\dotsc,a'_k).\]
This provides an alternative proof of Lemma \ref{lem:T2replete}, and shows that for $n=3$ we need only consider the operation $\ast^3_r$.
\end{remark}

We can already use this to deduce the number of replete subsemigroups of height at most $2$ in $T_n$.

\begin{lem}
\label{lem:height2count}
For any $n$, $T_n$ has exactly
\begin{equation}
2 \left(\binom{n}{0} + \binom{n}{1} + 18\binom{n}{2} \right) = 18n^2-16n+2
\end{equation}
subsemigroups of height at most $2$, all of which are replete. In particular, $T_3$ has $116$ replete subsemigroups of height at most $2$.
\end{lem}
\begin{proof}
By extension of Lemma \ref{lem:T2replete}, all subsemigroups containing at most two generators are replete, and we can count these either by inclusion-exclusion or by extending Table \ref{table:T2count}, adjusting the symmetry factor appropriately; this yields Table \ref{table:T3sub2count}. Summing up and doubling to account for the identity, we obtain the stated formula.
\begin{table}[ht]
\centering
\begin{tabular}{ l |c|c|c|c|c|c|}
$\bar{\al}(S)$ & $\emptyset$ & $\{\{a\}\}$ & $\{\{a,b\}\}$ & $\{\{a\},\{a,b\}\}$ & $\{\{a\},\{b\},\{a,b\}\}$ \\
\hline
Factor     & $\binom{n}{0}$ & $\binom{n}{1}$ & $\binom{n}{2}$ & $2\binom{n}{2}$ & $\binom{n}{2}$ \\
Count      & $1$ & $1$ & $9$ & $4$ & $1$
\end{tabular}
\caption{An accounting of the replete subsemigroups $S$ of $T_n$ of height at most $2$ which do not contain the identity, arranged by alphabets.}
\label{table:T3sub2count}
\end{table}
\end{proof}

Remark \ref{rmk:automaticcases} also enables us to identify all of the replete branch sets of height $3$. To begin with, there are $6$ minimal replete branch sets corresponding to the possible rightmost paths. By careful inspection of the operation ${\ast}^3_r$, any two rightmosts paths sharing their first or last generator form a replete subsemigroup, while any other pair generates one or two further rightmost paths. Finally, any possibility not accounted for by these must feature rightmost paths ending in each of the three possible generators, and hence must include all possible right branches. These possibilities are illustrated in Table \ref{table:hex1}.

For non-uniform subsemigroups of height $3$, we can deduce by considering ${\ast}_r$ how the presence of branches of smaller heights affect the height $3$ branches which appear: a tree of height $1$ forces (at least) one of two possible branch sets to appear, while a tree of height $2$ uniquely determines a branch set of height $3$, starting with the easiest case of subsemigroups of height at most $2$. Inspecting Table \ref{table:hex1}, we thus see that there are:
\begin{itemize}
	\item $12$ sets containing $[a,c,b \rangle$ or $[c,a,b \rangle$, which are closed under $-\ast_r [b\rangle$;
	\item $9$ sets containing $[a,c,b \rangle$, closed under $- \ast_r [a,b \rangle$;
	\item $6$ sets containing $[a,c,b \rangle$ and $[c,a,b \rangle$, closed under $- \ast_r [a,b \rangle$ and $- \ast_r [c,b \rangle$;
	\item $5$ sets containing $[a,b,c \rangle$ and $[a,c,b \rangle$, closed under $- \ast_r [b,c \rangle$ and $- \ast_r [c,b \rangle$;
	\item $3$ sets containing $[a,b,c \rangle$ and $[c,a,b \rangle$, closed under $- \ast_r [b,c \rangle$ and $- \ast_r [a,b \rangle$;
	\item $2$ sets containing $[a,b,c \rangle$ and $[c,b,a \rangle$, closed under $- \ast_r [b,c \rangle$ and $- \ast_r [b,a \rangle$.
\end{itemize}
These are the values appearing in the `$3$-branch sets' rows of Table \ref{table:T3count}.

To complete the count, recall from Theorem \ref{thm:uniformcount} that for each $A \in \bar{\al}(S)$ the leftmost and rightmost paths present in $S_A$ can be chosen independently, so we need only count the right branches for each possible alphabet and square the results before multiplying by the permutation factor, as we did for Table \ref{table:T3sub2count}. After doubling for the identity and including the trees of smaller height, we obtain the grand total.
\begin{prop}
\label{prop:T3subcount}
$T_3$ has $18,030$ replete subsemigroups.
\end{prop}
We will not explicitly go higher than $n=3$ in the present paper, but we can deduce the analogue of Lemma \ref{lem:height2count} without further work by simply summing the entries in the table.

\begin{lem}
\label{lem:height3count}
For any $n$, $T_n$ has exactly:
\begin{equation}
2 \left(\binom{n}{0} + \binom{n}{1} + 18\binom{n}{2} + 8957 \binom{n}{3} \right)
\end{equation}
replete subsemigroups of height at most $3$.
\end{lem}

\begin{table}[ht]
\centering
\begin{tabular}{c c c}
\begin{tikzpicture}
	%\draw (0:\rad) \foreach \x in {60,120,...,360} {  -- (\x:\rad) };
	\foreach \x/\l/\p in
     { 60/{B}/right,
      120/{A}/left,
      180/{F}/left,
      240/{E}/left,
      300/{D}/right,
      360/{C}/right
     }
     {\node[circle,draw,fill,label={\p:\l},color=blue,opacity=0.5] at (\x:\rad) {};
     \node[inner sep=1pt,circle,draw,fill,color=red] at (\x:\rad) {};
     }
\end{tikzpicture}
&
\begin{tikzpicture}
	\draw[color=green, line width=10pt, draw opacity=0.5, line cap = round] (120:\rad) -- (60:\rad) (240:\rad) -- (180:\rad) (360:\rad) -- (300:\rad);
	%\draw (0:\rad) \foreach \x in {60,120,...,360} {  -- (\x:\rad) };
	\foreach \x/\l/\p in
     { 60/{B}/right,
      120/{A}/left,
      180/{F}/left,
      240/{E}/left,
      300/{D}/right,
      360/{C}/right
     }
     \node[inner sep=1pt,circle,draw,fill,label={\p:\l},color=red] at (\x:\rad) {};
\end{tikzpicture}
&
\begin{tikzpicture}
   \draw[color=green, line width=10pt, draw opacity=0.5, line cap = round] (60:\rad) -- (360:\rad) (300:\rad) -- (240:\rad) (180:\rad) -- (120:\rad);
   %\draw (0:\rad) \foreach \x in {60,120,...,360} {  -- (\x:\rad) };
   \foreach \x/\l/\p in
     { 60/{B}/right,
      120/{A}/left,
      180/{F}/left,
      240/{E}/left,
      300/{D}/right,
      360/{C}/right
     }
     \node[inner sep=1pt,circle,draw,fill,label={\p:\l},color=red] at (\x:\rad) {};
\end{tikzpicture}
\\
\begin{tikzpicture}
   \draw[color=violet, line width=10pt, draw opacity=0.5, rounded corners] (120:\rad) -- (60:\rad) -- (360:\rad) -- cycle (240:\rad) -- (180:\rad) -- (300:\rad) -- cycle;
   %\draw (0:\rad) \foreach \x in {60,120,...,360} {  -- (\x:\rad) };
   \foreach \x/\l/\p in
     { 60/{B}/right,
      120/{A}/left,
      180/{F}/left,
      240/{E}/left,
      300/{D}/right,
      360/{C}/right
     }
     \node[inner sep=1pt,circle,draw,fill,label={\p:\l},color=red] at (\x:\rad) {};
   \node[inner sep = 1pt, circle,draw,fill,color=black] at (60:\rad) {};
   \node[inner sep = 1pt, circle,draw,fill,color=black] at (240:\rad) {};
\end{tikzpicture}
&
\begin{tikzpicture}
   \draw[color=violet, line width=10pt, draw opacity=0.5, rounded corners] (300:\rad) -- (60:\rad) -- (360:\rad) -- cycle (240:\rad) -- (180:\rad) -- (120:\rad) -- cycle;
   %\draw (0:\rad) \foreach \x in {60,120,...,360} {  -- (\x:\rad) };
   \foreach \x/\l/\p in
     { 60/{B}/right,
      120/{A}/left,
      180/{F}/left,
      240/{E}/left,
      300/{D}/right,
      360/{C}/right
     }
     \node[inner sep=1pt,circle,draw,fill,label={\p:\l},color=red] at (\x:\rad) {};
   \node[inner sep = 1pt, circle,draw,fill,color=black] at (360:\rad) {};
   \node[inner sep = 1pt, circle,draw,fill,color=black] at (180:\rad) {};
\end{tikzpicture}
&
\begin{tikzpicture}
   \draw[color=violet, line width=10pt, draw opacity=0.5, rounded corners] (300:\rad) -- (240:\rad) -- (360:\rad) -- cycle (60:\rad) -- (180:\rad) -- (120:\rad) -- cycle;
   \foreach \x/\l/\p in
     { 60/{B}/right,
      120/{A}/left,
      180/{F}/left,
      240/{E}/left,
      300/{D}/right,
      360/{C}/right
     }
     \node[inner sep=1pt,circle,draw,fill,label={\p:\l},color=red] at (\x:\rad) {};
   \node[inner sep = 1pt, circle,draw,fill,color=black] at (300:\rad) {};
   \node[inner sep = 1pt, circle,draw,fill,color=black] at (120:\rad) {};
\end{tikzpicture}
\\
\begin{tikzpicture}
   \draw[color=orange, line width=10pt, draw opacity=0.5, rounded corners] (240:\rad) -- (300:\rad) -- (120:\rad) -- (180:\rad) -- cycle;
   \foreach \x/\l/\p in
     { 60/{B}/right,
      120/{A}/left,
      180/{F}/left,
      240/{E}/left,
      300/{D}/right,
      360/{C}/right
     }
     \node[inner sep=1pt,circle,draw,fill,label={\p:\l},color=black] at (\x:\rad) {};
   \node[inner sep = 1pt, circle,draw,fill,color=red] at (120:\rad) {};
   \node[inner sep = 1pt, circle,draw,fill,color=red] at (300:\rad) {};
\end{tikzpicture}
&
\begin{tikzpicture}
   \draw[color=orange, line width=10pt, draw opacity=0.5, rounded corners] (120:\rad) -- (180:\rad) -- (360:\rad) -- (60:\rad) -- cycle;
   \foreach \x/\l/\p in
     { 60/{B}/right,
      120/{A}/left,
      180/{F}/left,
      240/{E}/left,
      300/{D}/right,
      360/{C}/right
     }
     \node[inner sep=1pt,circle,draw,fill,label={\p:\l},color=black] at (\x:\rad) {};
   \node[inner sep = 1pt, circle,draw,fill,color=red] at (180:\rad) {};
   \node[inner sep = 1pt, circle,draw,fill,color=red] at (360:\rad) {};
\end{tikzpicture}
&
\begin{tikzpicture}
   \draw[color=orange, line width=10pt, draw opacity=0.5, rounded corners] (240:\rad) -- (300:\rad) -- (360:\rad) -- (60:\rad) -- cycle;
   \foreach \x/\l/\p in
     { 60/{B}/right,
      120/{A}/left,
      180/{F}/left,
      240/{E}/left,
      300/{D}/right,
      360/{C}/right
     }
     \node[inner sep=1pt,circle,draw,fill,label={\p:\l},color=black] at (\x:\rad) {};
   \node[inner sep = 1pt, circle,draw,fill,color=red] at (60:\rad) {};
   \node[inner sep = 1pt, circle,draw,fill,color=red] at (240:\rad) {};
\end{tikzpicture}
\\
&
\begin{tikzpicture}
   \draw[color=gray, line width=10pt, draw opacity=0.5, rounded corners] (240:\rad) -- (300:\rad) -- (360:\rad) -- (60:\rad) -- (120:\rad) -- (180:\rad) -- cycle;
   \foreach \x/\l/\p in
     { 60/{B}/right,
      120/{A}/left,
      180/{F}/left,
      240/{E}/left,
      300/{D}/right,
      360/{C}/right
     }
     \node[inner sep=1pt,circle,draw,fill,label={\p:\l},color=black] at (\x:\rad) {};
   \node[inner sep = 1pt, circle,draw,fill,color=red] at (60:\rad) {};
   \node[inner sep = 1pt, circle,draw,fill,color=red] at (240:\rad) {};
   \node[inner sep = 1pt, circle,draw,fill,color=red] at (180:\rad) {};
\end{tikzpicture}
&
\end{tabular}
\caption{A diagrammatic representation of the $22$ (inhabited) replete right branch sets. In each hexagon, nodes A, B, C, D, E, F respectively represent the minimal replete branch sets $[a,b,c\rangle$, $[a,c,b\rangle$, $[c,a,b\rangle$, $[c,b,a\rangle$, $[b,c,a\rangle$, $[b,a,c\rangle$ and each connected shape represents the replete set of branches generated by the minimal sets indicated in red which it contains.}
\label{table:hex1}
\end{table}

%\textcolor{orange}{CONSIDER PERMUTAHEDRA: We can construct the edges (two-vertex sets) in Table \ref{table:hex1} as the transpositions of the first or second pair of elements.}

\begin{landscape}

\begin{table}[ht]
\centering
\DivWidth{\colOne}{$\{\{a,b\}\}$}{2}
\DivWidth{\colTwo}{$\{\{a\},\{a,b\}\}$}{2}
\begin{tabu}{ l |c|c|C{\colOne}|C{\colOne}|C{\colTwo}|C{\colTwo}|c|C{\colTwo}|C{\colTwo}|}
$\bar{\al}(S) \backslash \{[3]\}$ & $\emptyset$ & $\{\{a\}\}$ & \multicolumn{2}{c|}{$\{\{a,b\}\}$} & \multicolumn{2}{c|}{$\{\{a\},\{a,b\}\}$} & $\{\{a\},\{b\},\{a,b\}\}$ & \multicolumn{2}{c|}{$\{\{a\},\{b,c\}\}$} 
\\
\hline
Factor          & $1$  & $3$  & \multicolumn{2}{c|}{$3$} & \multicolumn{2}{c|}{$6$} & $3$ & \multicolumn{2}{c|}{$3$} \\
\hline
\rowfont{\color{gray}}
$2$-branches    & $0$  & $0$  & $1$ & $2$ & $1$ & $2$ & $2$ & $1$ & $2$ \\
Count           & $1$  & $1$  & $2$ & $1$ & $1$ & $1$ & $1$ & $2$ & $1$ \\ 
$3$-branch sets & $22$ & $12$ & $9$ & $5$ & $9$ & $5$ & $5$ & $3$ & $1$ \\
\hline
Replete $S$     & $484$ & $432$ & \multicolumn{2}{c|}{$1587$} & \multicolumn{2}{c|}{$1176$} & $75$ & \multicolumn{2}{c|}{$147$}
\end{tabu}
\\
\vspace{6pt}
\DivWidth{\colOne}{$\{\{a\},\{b\},\{a,b\},\{b,c\}\}$}{2}
\begin{tabu}{ l |c|c|c|c|c|c|c|c|c|c|c|C{\colOne}|C{\colOne}|}
$\bar{\al}(S) \backslash \{[3]\}$ & \multicolumn{4}{c|}{$\{\{a,b\},\{b,c\}\}$} & \multicolumn{4}{c|}{$\{\{a\},\{a,b\},\{b,c\}\}$} & \multicolumn{3}{c|}{$\{\{b\},\{a,b\},\{b,c\}\}$} & \multicolumn{2}{c|}{$\{\{a\},\{b\},\{a,b\},\{b,c\}\}$} \\
\hline
Factor          & \multicolumn{4}{c|}{$3$} & \multicolumn{4}{c|}{$6$}
                & \multicolumn{3}{c|}{$3$} & \multicolumn{2}{c|}{$6$} \\
\hline
\rowfont{\color{gray}}
$2$-branches    & $[a,b\rangle$   & $[b,a\rangle,2$ & $[b,a\rangle$ & $2$
                & $[b,a\rangle,2$ & $[b,a\rangle$   & $[b,a\rangle$ & $2$             
                & $[a,b\rangle$   & $[a,b\rangle$   & $2$           & $2$             & $2$ \\
\rowfont{\color{gray}}
                & $[c,b\rangle$   & $[c,b\rangle$   & $[b,c\rangle$ & $[b,c\rangle,2$ 
                & $[c,b\rangle$   & $[b,c\rangle$   & $2$           & $[b,c\rangle,2$
                & $[c,b\rangle$   & $2$             & $2$           & $[c,b\rangle$   & $2$ \\
Count           & $1$             & $4$             & $1$           & $3$ 
                & $2$             & $1$             & $1$           & $2$ 
                & $1$             & $2$             & $1$           & $1$             & $1$ \\
$3$-branch sets & $6$             & $3$             & $2$           & $1$
                & $3$             & $2$             & $1$           & $1$ 
                & $6$             & $3$             & $1$           & $3$             & $1$ \\
\hline
Replete $S$ & \multicolumn{4}{c|}{$1587$} & \multicolumn{4}{c|}{$726$}
                & \multicolumn{3}{c|}{$507$} & \multicolumn{2}{c|}{$96$}
\end{tabu}
\\
\vspace{6pt}
\DivWidth{\colOne}{$\{\{a\},\{b\},\{a,b\},\{b,c\},\{a,c\}\}$}{3}
\begin{tabu}{ l |c|c|c|c|c|c|c|c|c|C{\colOne}|C{\colOne}|C{\colOne}|c|}
$\bar{\al}(S) \backslash \{[3]\}$ & \multicolumn{4}{c|}{$\{\{a,b\},\{b,c\},\{a,c\}\}$} & \multicolumn{5}{c|}{$\{\{a\},\{a,b\},\{b,c\},\{a,c\}\}$} & \multicolumn{3}{c|}{$\{\{a\},\{b\},\{a,b\},\{b,c\},\{a,c\}\}$} & all \\
\hline
Factor          & \multicolumn{4}{c|}{$1$} & \multicolumn{5}{c|}{$3$} & \multicolumn{3}{c|}{$3$} & $1$ \\
\hline
\rowfont{\color{gray}}
$2$-branches    & $[a,b\rangle$ & $[b,a\rangle$   & $[b,a\rangle$ & $2$ 
                & $[b,a\rangle$ & $[b,a\rangle$   & $2$           & $2$             & $2$ 
                & $2$           & $2$             & $2$           & $2$ \\
\rowfont{\color{gray}}
                & $1,2$         & $[c,b\rangle,2$ & $2$           & $1,2$
                & $1$           & $2$             & $[c,b\rangle$ & $[b,c\rangle,2$ & $1,2$
                & $[c,b\rangle$ & $2$             & $[c,b\rangle$ & $2$ \\
\rowfont{\color{gray}}
                & $[a,c\rangle$ & $[a,c\rangle$   & $[c,a\rangle$ & $2$
                & $[c,a\rangle$ & $[c,a\rangle$   & $[c,a\rangle$ & $[c,a\rangle$   & $2$
                & $[c,a\rangle$ & $[c,a\rangle,2$ & $2$           & $2$   \\
Count           & $9$           & $8$             & $3$           & $7$
                & $2$           & $1$             & $2$           & $4$             & $3$
                & $1$           & $2$             & $1$           & $1$   \\
$3$-branch sets & $2$           & $1$             & $1$           & $1$
                & $2$           & $1$             & $2$           & $1$             & $1$
                & $2$           & $1$             & $1$           & $1$ \\
\hline
Replete $S$ & \multicolumn{4}{c|}{$1296$} & \multicolumn{5}{c|}{$768$} & \multicolumn{3}{c|}{$75$} & $1$
\end{tabu}
\caption{An accounting of the replete branch sets of $T_3$ whose alphabets include $[3]=\{a,b,c\}$ (omitted from the column headings). The `Factor' accounts for the alphabets obtained by permuting the generators. The `$2$-branches' rows specify the types of $2$-branch assumed to be present (or simply the number if the possibilities are symmetric); the `Count' is the number of ways of selecting such a $2$-branch set, accounting for symmetries of $[n]$ which fix $\bar{\al}(S)$. The `$3$-branch sets' row counts the number of sets from Table \ref{table:hex1} which are compatible with the given ($1$- and) $2$-branches.}
\label{table:T3count}
\end{table}

\end{landscape}

\clearpage

\section{Mirigs}
\label{sec:mirig}

With all of this theory in place, we can at last move on to mirigs.

\subsection{Characteristic}
\label{ssec:char}

Just as in ring theory, an important feature of a rig is the minimal subring containing $0$ and $1$. For rings, one number suffices to characterize this subring, but for rigs we need two, corresponding to an indexing of the quotient rigs of $\Nbb$.

For $m\geq 0, n\geq 1$, we denote by $\Nbb_{m,n}$ the rig whose underlying set is $\{0,1,\dotsc,m+n-1\}$, where addition and multiplication are as in $\Nbb$ but with any result $k$ greater than $m+n-1$ reduced modulo $n$ into the range $m \leq k \leq m+n-1$. We omit the proof that all proper quotients of $\Nbb$ are of the form $\Nbb_{m,n}$ for some $m,n$. This rig is a ring if and only if $m = 0$, with $\Nbb_{0,n} \cong \Zbb/n\Zbb$. There exists a rig homomorphism $\Nbb_{m,n} \to \Nbb_{m',n'}$ if and only if $m' \leq m$ and $n' \mid n$.

\begin{exa}
\label{exa:R0}
Recall that the free rig on the empty set (zero generators) is isomorphic to the rig of natural numbers, $\Nbb$. The free mirig on zero generators is therefore the quotient of $\Nbb$ by the congruence generated by $n \approx n^2$. Examining the first few generating relations, we in particular have $2 \approx 4$, which implies that $4+k \approx 2+k$ for all $k \geq 0$. Applying this relation inductively, every even number greater than $2$ is identified with $2$ and every odd number greater than $2$ is identified with $3$. By a parity argument, the remaining basic relations create no further identifications; in other words, the congruence is generated by the basic relation $2 \approx 4$. Thus the free mirig on zero generators is isomorphic to $\Nbb_{2,2}$, and in notation to be established in Section \ref{ssec:direct}, we have:
\[|R_0| = 4.\]

Incidentally, by a similar argument one can show that the free rig subject to the equation $n = n^3$ on zero generators is isomorphic to $\Nbb_{2,6}$.
\end{exa}

\begin{defn}
We say a rig has \textbf{characteristic $(m,n)$} if the subrig generated by $1$ is isomorphic to $\Nbb_{m,n}$. We extend this with the convention that a rig has \textbf{characteristic $(\infty,0)$} if the subrig generated by $1$ is isomorphic to $\Nbb$, which is consistent with the conditions for the existence of a rig homomorphism to $\Nbb_{m,n}$ for every $m,n$.
\end{defn}

\begin{remark}
There is not a strongly established convention for expressing the characteristic of rigs in analogy with characteristic of rings. In the 1990s, Alarc\'{o}n and Anderson use the notation $B(n,i)$ for the rig which we would denote $\Nbb_{i,n-i}$, \cite[Example 3]{char1}; they use the corresponding convention for characteristic. Even earlier, in the 1970s, Weinert used a numbering convention offset from ours by $1$, \cite{char2}. We chose the present convention because it simplifies the conditions for the existence of rig homomorphisms between the rigs $\Nbb_{m,n}$.
\end{remark}

Whenever we have a rig homomorphism $R \to R'$, this restricts to a homomorphism between the respective subrigs generated by $1$. Since any mirig admits a (unique) rig homomorphism from the free mirig on zero generators, we can deduce the following from Example \ref{exa:R0}:

\begin{lem}
\label{lem:char}
A mirig has characteristic $(2,2)$, $(1,2)$, $(0,2)$, $(2,1)$, $(1,1)$ or $(0,1)$. In particular,
\begin{itemize}
	\item A free mirig has characteristic $(2,2)$.
	\item A mirig has characteristic $(0,2)$ if and only if it is a Boolean ring.
	\item A mirig has characteristic $(0,1)$ if and only if it is (isomorphic to) the degenerate rig with $0=1$.
\end{itemize}
\end{lem}
One can check manually that the quotients of $\Nbb$ with these characteristics are all mirigs, which demonstrates that all signatures in Lemma \ref{lem:char} are possible.

\subsection{Mirigs need not be commutative}

Returning to Jean-Baptiste Vienney's original question, we are already more than equipped to provide counterexamples to the conjecture that mirigs are commutative. The following was proposed by Tim Campion in the original Zulip discussion.

\begin{exa}[Campion's idempotent monoid to mirig construction]
\label{exa:idmon2mirig}
Let $M$ be an idempotent monoid, viewing the monoid operation on $M$ as multiplication and denoting the identity element by $1$. Let $M' = M \sqcup \{0,2\}$. Extend the monoid structure on $M$ to a structure on $M'$ by setting
\begin{align*}
	2 \cdot 0 = 0 \cdot 2 &= 2\\
	2 \cdot 2 &= 2 \\
	0 \cdot 0 &= 0
\end{align*}
and for each $m \in M$,
\begin{align*}
	m \cdot 2 = 2 \cdot m &= 2 \\
	m \cdot 0 = 0\cdot m &= 0.
\end{align*}
Then define addition on $M'$ by making $0$ the additive identity, making $2$ an absorbing element (so $x + 2 = 2 + x = 2$ for all $x \in M'$), and setting $m + n = 2$ for $m,n \in M$. Then $M$ is a mirig, which is commutative if and only if $M$ is.

Using the free idempotent monoid on two generators from Example \ref{exa:M2}, we obtain an instance of a non-commutative mirig with just $9$ elements. Observe that all mirigs produced by this construction have characteristic $(2,1)$, and if we apply it to the trivial monoid we recover $\Nbb_{2,1}$.
\end{exa}

In describing this class of examples, Tim Campion drew our collective attention to the OEIS entry on idempotent monoids \cite{oeis}, which in turn brought us to Green and Rees' work on free idempotent monoids. The present author observed in light of this that the free mirig on finitely many generators must be finite (see Lemma \ref{lem:Rnfinite}), and wondered exactly how large this mirig is; this calculation turned out to be challenging even in the fairly simple case of two generators. John Baez \href{https://mathstodon.xyz/@johncarlosbaez/109544916566242548}{posted this problem} on the distributed social networking platform Mastodon in December 2022. There followed a collaborative puzzle-solving effort, some details of which are recorded below. Further discussion can be found at the Zulip channel linked above and in Baez's \href{https://johncarlosbaez.wordpress.com/2022/12/21/free-idempotent-rigs-and-monoids/}{blog post} on the subject.

\subsection{Free mirigs}
\label{ssec:direct}

Let us establish some notation for free mirigs. We write $S_n$ for the free rig on $[n]$. Let ${\approx}$ be the congruence on $S_n$ generated by the relation $r \approx r^2$, so that $R_n := S_n/{\approx}$ is the free mirig on $[n]$.

We have a natural inclusion $F_n \subseteq S_n$ on the singleton sums, and the restriction of ${\approx}$ to this subset coincides with ${\sim}$. However, the congruence on $S_n$ generated by ${\sim}$ is strictly smaller than ${\approx}$; we shall get an indication of the how much smaller it is in Lemma \ref{lem:simsim}. By mild abuse of notation, we write ${\sim}$ for the congruence on $S_n$ generated by the congruence ${\sim}$ on $F_n$. Then (the additive monoid of) $S_n/{\sim}$ is isomorphic to the free commutative monoid on $F_n/{\sim} = M_n$; we accordingly denote the quotient rig $S_n/{\sim}$ by $\Nbb[M_n]$.

Between ${\sim}$ and ${\approx}$ we have the congruence ${\simeq}$ obtained by adding the basic relation $2 \simeq 4$. This congruence identifies elements of $\Nbb[M_n]$ if their coefficients are equivalent in $\Nbb_{2,2}$. As such, we denote $S_n/{\simeq}$ by $\Nbb_{2,2}[M_n]$. In summary, we have quotient maps,
\[ S_n \twoheadrightarrow \Nbb[M_n] \twoheadrightarrow \Nbb_{2,2}[M_n] \twoheadrightarrow R_n.\]

By inspection of the penultimate map in the sequence, we can deduce the result which inspired the Mastodon discussion.

\begin{lem}
\label{lem:Rnfinite}
Elements of $R_n$ are presented by elements of $\Nbb_{2,2}[M_n]$. In particular, $|R_n| \leq |\Nbb_{2,2}[M_n]| = 4^{|M_n|}$ is finite for all $n$.
\end{lem}

A theme of the approaches taken by the Mastodon participants to compute $|R_2|$ was to consider the set of representatives with small coefficients from Lemma \ref{lem:Rnfinite} and iteratively apply congruences to determine which of these representatives are identified by ${\approx}$. Simon Frankau did this (\href{https://github.com/simon-frankau/two-generator-idempotent-rigs}{in Rust}, \cite{Frankau}) by inductively generating the congruence ${\simeq}$. Greg Egan took a similar approach (\href{https://github.com/nagegerg/IdempotentRig}{in C++}, \cite{Egan}). It turns out that the procedure for generating the congruence can be simplified thanks to the following result, based on a \href{https://schelling.pt/@alexthecamel/109550926277189496}{comment of Alex Gunning} in the Mastodon discussion.

\begin{lem}[Gunning's Lemma]
\label{lem:simsim}
The congruence ${\approx}$ on $S_n$ is generated by the elementary relations $2 \approx 4$, $w \approx w^2$ and $u + v \approx u + uv + vu + v$ for $u,v,w$ monomials (individual words).
\end{lem}
\begin{proof}
Let ${\approx'}$ be the congruence generated by the given relations. We have ${\approx'} \subseteq {\approx}$ by inspection. We shall show that ${\approx'} \supseteq {\approx}$; it suffices to prove that if $r$ is a sum with $N$ terms, then $r \approx' r^2$. The cases $N = 0,1,2$ are almost satisfied by assumption: we know that the first relation implies $k \approx' k^2$ for all constants $k$ by the argument of Example \ref{exa:R0}; combining this with the second generating relation we have $kw \approx' k^2w^2$ for all words $w$ and coefficients $k$, and similarly for two-term sums. So let $r = \sum_{i=0}^{N} k_i w_i$ be an element of $S_n$ with $N \geq 3$ terms. For each pair of indices $i < j$, we may use the two-term sum relation to substitute $k_iw_i + k_jw_j$ for $k_iw_i + k_iw_ik_jw_j + k_jw_jk_iw_i + k_jw_j$ (one pair at a time) and then substitute all of the original terms for their squares using the one-term relations. The final expression is $r^2$, whence $r \approx' r^2$, as required.
\end{proof} 

Alex Gunning calculated $|R_2|$ (\href{https://github.com/agunning/freerig}{in Python}, \cite{Gunning}) by generating a graph whose nodes are elements of $\Nbb_{2,2}[M_2]$ and whose edges are of the form $w + x(u+v)y \to w + x(u + uv + vu + v)y$. Lemma \ref{lem:simsim} exactly says that the connected components of this graph correspond to ${\approx}$ equivalence classes. While the same strategy may be feasible for a few more generators, the growth rate of $|M_n|$ (and the corresponding explosion in size of $4^{|M_n|}$) is prohibitive.

\begin{exa}
\label{exa:R1}
We can use generic instances of Gunning's Lemma to improve the bound of Lemma \ref{lem:Rnfinite} and hence reduce the size of the computations. For example, observe that for any $x$, we have $1+x \approx (1+x)^2 \approx 1 + 3x$. As such, by separately accounting for elements of $\Nbb_{2,2}[M_n]$ having $0$ coefficient at the identity and elements having coefficient at least $1$ we can restrict the coefficients in the latter case to $0$, $1$ or $2$. We thus obtain a bound of:
\begin{equation}
	\label{eq:betterbound}
	|R_n| \leq 4^{|M_n|-1} + 3 \times 3^{|M_n|-1}.
\end{equation}
When $n=1$ this gives a bound of $13$, and it is straightforward to manually check that there are no further identifications. That is,
\[|R_1| = 13.\]
\end{exa}

\begin{remark}
For larger numbers of generators, even the bound provided by \eqref{eq:betterbound} is loose: although for $n = 2$ it reduces the upper bound from $16,384$ to $6,283$, we shall eventually see that the actual value of $|R_2|$ is an order of magnitude smaller still. While it is possible to produce further \textit{ad hoc} improvements in a similar vein, we decided to take a more systematic approach to obtain exact values in the present article.
\end{remark}

Computational efficiency aside, all three of the cited participants in the Mastodon discussion agreed on a final count of $|R_2| = 284$. In the remainder of the present article, we shall extend our preparatory work on idempotent monoids from Section \ref{sec:idMon} to reaffirm this value and then calculate $|R_3|$ by hand. The author hopes that these methods might eventually be adapted into a program capable of computing $|R_n|$ for larger values of $n$ (we expect this to be feasible up to $n=8$), so that the sequence $|R_n|$ can be added to the OEIS alongside the entry counting the cardinalities of free idempotent monoids, \cite{oeis}.

\subsection{Thickets}
\label{ssec:forests}

From here onwards, we pass across the isomorphism $M_n \cong T_n$, although it will occasionally be convenient for concrete examples to return to (equivalence classes of) words in $M_n$. The sequence of quotients from earlier now becomes,
\[ S_n \twoheadrightarrow \Nbb[T_n] \twoheadrightarrow \Nbb_{2,2}[T_n] \twoheadrightarrow R_n.\]

\begin{defn}
\label{defn:thickets}
We shall refer to elements of $\Nbb[T_n]$ as \textbf{forests} (of labelled, rooted, binary trees), and correspondingly refer to elements of $\Nbb_{2,2}[T_n]$ as \textbf{thickets}.
\end{defn}

We shall primarily employ thickets as representatives for elements of $R_n$. The goal of this section is to derive properties of thickets which are invariant under the congruence ${\approx}$ such that each element of $R_n$ is uniquely determined by these invariants. We begin with two definitions on thickets.

\begin{defn}
\label{def:apparity}
The \textbf{apparity} function\footnote{The term `parity' is so well established for reduction modulo $2$ that I needed to derive a new term from it.} $P:\Nbb_{2,2}[T_n] \to \Nbb_{2,2}$ sends a thicket $f$ to the sum of its coefficients in $\Nbb_{2,2}$. This is the rig homomorphism generated by sending all generators to $1 \in \Nbb_{2,2}$.

We extend $\al:T_n \to \Pcal([n])$ to an \textbf{alphabet} function $\al:\Nbb_{2,2}[T_n] \to \Pcal([n])$ by taking the union of the function's values over the trees appearing as summands in a given thicket.
\end{defn}

\begin{lem}
\label{lem:Palm}
The maps $P$ and $\al$ descend to functions on $R_n$; we use the same names for the resulting functions.
\end{lem}
\begin{proof}
By inspection, these functions are invariant across the basic relation $u+v \approx u+uv+vu+v$ identified in Lemma \ref{lem:simsim}.
\end{proof}

\begin{defn}
\label{def:summand}
Let $f,f' \in \Nbb[T_n]$ or $\Nbb_{2,2}[T_n]$. We say $f'$ is a \textbf{summand} of $f$ if there exists $f''$ with $f = f' + f''$. Viewing $T_n$ as a subset of $\Nbb[T_n]$ (resp. $\Nbb_{2,2}[T_n]$), for $t \in T_n$ we write $t \in f$ when $t$ is a summand of $f$.
\end{defn}
Two forests or thickets represent the same element of $R_n$ whenever they are related by a series of substitutions of factors of summands by their squares or square roots.

\begin{lem}
\label{lem:amalgamation}
Let $f,f',f''$ be thickets and suppose that we can get from $f$ to $f'$ by a series of applications of expansions $u+v \mapsto u+uv+vu+v$ to factors of summands, and similarly from $f$ to $f''$. Then there exists $f'''$ reachable from $f'$ and $f''$ by expansions.
\end{lem}
\begin{proof}
Since we apply only expansions, $f$ is a summand of $f'$, and so we may expand $f'$ to the thicket which may be informally denoted $f'+f''-f$. Reversing the analysis, this thicket is also accessible from $f''$.
\end{proof}

\begin{prop}
\label{prop:confluent}
Each element of $R_n$ admits a maximal representative thicket: a thicket $f$ such that all $f'$ with $f' \approx f$ are summands of $f$.
\end{prop}
\begin{proof}
Preorder $\Nbb_{2,2}[T_n]$ by expansion: $f \leq f'$ if $f=f'$ or there is a sequence of expansions which may be applied to $f$ to obtain $f'$; note that this is not a partial order relation (see Example \ref{exa:maxNotUniq}, below). Two elements $f,f'$ represent the same element of $R_n$ if and only if they are related by a zigzag in this poset, meaning $f=f_1 \leq f_2 \geq f_3 \leq \cdots \geq f_n = f'$. Applying Lemma \ref{lem:amalgamation} iteratively, we conclude that any zigzag can be reduced to a single cospan: two elements are related if and only if they have a mutual upper bound in the poset. Since there are only finitely many thickets, this means that each component of the preorder contains a maximal element, as required.
\end{proof}

\begin{exa}
\label{exa:maxNotUniq}
Consider $2ab+2aba+3bab+2ba \in \Nbb_{2,2}[M_2]$. Applying the expansion $ab+ba \mapsto ab+aba+bab+ba$, we obtain $2ab+3aba+2ab+2ba$; applying this expansion again returns us to the original element. Both thickets are maximal representatives for the same element of $R_n$.
\end{exa}

\begin{remark}
From the perspective of \textit{rewriting theory}, Proposition \ref{prop:confluent} is quite unusual: one commonly encounters rewriting systems (formal systems of directed rules for manipulating formal languages) in which the rewriting rules \textit{reduce} (some notion of) the size of terms, so that in a confluent system one arrives at minimal terms. Rewriting systems producing maximal terms are rare because an upper bound on sizes, such as the finiteness employed in the proof of Proposition \ref{prop:confluent}, is needed to guarantee their existence.
\end{remark}

\begin{defn}
\label{def:support}
The \textbf{support} of a forest or thicket $f$ is $\supp(f):=\{t \mid t \in f\}$ (the set of summands obtained by ignoring the coefficients).
\end{defn}
While the support of a thicket is clearly not invariant under ${\approx}$, the support \textit{is} invariant amongst equivalent \textit{maximal} thickets.

\begin{lem}
\label{lem:suppinvariant}
Suppose $f \approx f' \in \Nbb_{2,2}[T_n]$ are two maximal thickets. Then $\supp(f) = \supp(f')$ is a subsemigroup of $T_n$.
\end{lem}
\begin{proof}
The support of $f$ contains the support of any summand of $f$, including $f'$ by the assumption of maximality. Conversely, $\supp(f) \subseteq \supp(f')$. To show that $\supp(f)$ is a subsemigroup, suppose we are given $s,t \in f$. Then either $s=t=st \in f$ or $s+t+st+ts$ is a summand of $f$, so $st,ts \in f$.
\end{proof}

The subsemigroup $\supp(f)$ doesn't contain enough information to recover the (equivalence class of the) maximal thicket $f$, and conversely not every subsemigroup of $T_n$ is the support of a maximal thicket. We shall close this gap shortly.

\begin{defn}
\label{def:uniform}
We say a thicket $f \in \Nbb_{2,2}[T_n]$ is \textbf{uniform} if $f$ is not $0$ and whenever $f = f_1 + f_2$, we have $\al(f_1) = \al(f_2)$ or $f_1f_2 = 0$.
\end{defn}
The maximal uniform thickets are those whose support is a uniform subsemigroup of $T_n$. Accordingly, we can decompose any thicket $f$ into a sum of uniform thickets in a canonical way:
\begin{equation}
\label{eq:uniform}
f = \sum_{A \subseteq \al(f)} f_A,
\end{equation}
where $f_A$ is (either zero or) the summand consisting of all trees $t \in f$ with $\al(t) = A$. Some of the most important uniform summands of $f$ turn out to be those consisting of lonely trees.

\begin{lem}
\label{lem:lonelyhypothesis}
Let $f$ be a maximal thicket and $S := \supp(f)$. Then $P(f_A) = 1$ implies that $A$ is minimal in $\bar{\al}(S)$.
\end{lem}
\begin{proof}
We already know from Proposition \ref{prop:subsemibound} that $A$ must be union-irreducible, since $P(f_A) = 1$ implies that $t \in S_A$ is a lonely tree. Suppose there exists $A' \subsetneq A$ in $\bar{\al}(S)$, and let $t' \in S_{A'}$. Then $t+t'$ is a summand of $f$, whence $t+t'+tt'+t't$ is too, but the latter pair of terms also lie in $S_A$, whence $P(f_A) = 2$ or $3$ by maximality, a contradiction.
\end{proof}

\begin{remark}
\label{rem:treecoeffs}
Note that if $f$ is a maximal thicket and $P(f_A)$ is $2$ or $3$, then any $t \in f_A$ must have coefficient $2$ or $3$ also. Indeed, there must exist some other tree $s \in f_A$ (possibly another copy of $t$) to which we may apply the expansions $t+s \mapsto t+s+st+ts \mapsto 2t+2s+ts+st$ to conclude by maximality that $f$ must have $2t+2s$ as a summand, so $t$ has coefficient $2$ or $3$.
\end{remark}

\begin{lem}
\label{lem:fixedparity}
Given equivalent maximal thickets $f \approx f'$, we have $P(f_A) = P(f'_A)$ for each $A \subseteq [n]$.
\end{lem}
\begin{proof}
Let $S := \supp(f)$. Consider an expansion $x(u+v)y \mapsto x(u+v+uv+vu)y$ which can be applied to $f$. If $\al(xuy)=\al(xvy)$ is minimal in $\bar{\al}(S)$, then for the expansion to be applicable to $f$ we must have $P(f_{\al(xuv)})$ being two or three, and this expansion adds two trees on the same alphabet to $f$, which does not affect $P(f_A)$ (nor the parity of any other uniform summand of $f$).

Otherwise, by Lemma \ref{lem:lonelyhypothesis} $P(f_{\al(xuvy)})$ is equal to two or three since $\al(xuvy)$ is not minimal in $\bar{\al}(S)$, and hence adding two trees on this alphabet does not change the parity of the uniform summands.

Since $f'$ is an expansion of $f$ and vice versa, the result follows.
\end{proof}

Following Lemmas \ref{lem:lonelyhypothesis} and \ref{lem:fixedparity}, we call a lonely tree with coefficient $1$ in a maximal thicket $f$ a \textbf{straggler} and say $f$ is \textbf{without stragglers} if it has none, meaning $P(f_A) \neq 1$ for all $A \subseteq [n]$. We can decompose any maximal $f \in \Nbb_{2,2}[T_n]$ as $f^{\mathrm{str}}+f^{\mathrm{w/o}}$ where the first term is the sum of the stragglers of $f$ and the latter is the maximal summand of $f$ without stragglers.

\begin{cor}
\label{cor:unstraggled}
Let $f \approx f'$ be maximal thickets. Then $f^{\mathrm{str}}={f'}^{\mathrm{str}}$ and $f^{\mathrm{w/o}} \approx {f'}^{\mathrm{w/o}}$. Moreover, the latter thickets are maximal.
\end{cor}
\begin{proof}
That $f$ and $f'$ must have the same stragglers follows from Lemma \ref{lem:fixedparity} and Lemma \ref{lem:suppinvariant}.

Suppose we are given an expansion from $f^{\mathrm{w/o}}$ involving a straggler, say $x(u+v)y \mapsto x(u+v+uv+vu)y$ with $xuv$ a straggler. Then by maximality of $f$, $x(uv+vu)y$ is a summand of $f^{\mathrm{w/o}}$, and we have expansions $x(uv+vu)y \mapsto x(uv+vu+uvu+vuv)y \mapsto x(2uv+2vu + uvu+vuv)y \mapsfrom x(2uv+2vu)y$. It follows that any expansion from $f$ to $f'$ can be transformed into a witness that $f^{\mathrm{w/o}} \approx {f'}^{\mathrm{w/o}}$, and conversely the fact that any expansion of $f^{\mathrm{w/o}}$ can be transformed into an expansion of $f$ ensures that $f^{\mathrm{w/o}}$ is maximal.
\end{proof}

The fact that $\supp(f)$ is a subsemigroup for maximal $f$ accounts for closure under expansions of the form $u+v \mapsto u+v+uv+vu$, but it does not control expansions of factors of thickets. Having decomposed our maximal thickets, we can characterize the subsemigroups arising as there supports.

\begin{lem}
\label{lem:inverseimagemonoids}
Let $f \in \Nbb_{2,2}[T_n]$ be a maximal thicket. Then $\supp(f^{\mathrm{w/o}})$ is a replete subsemigroup of $T_n$.
\end{lem}
\begin{proof}
If $u,v \in x\backslash \supp(f^{\mathrm{w/o}}) / y$, then $xuy+xvy$ is a summand of $f$, so $x(u+v+uv+vu)y$ is a summand of $f$ by maximality and $uv,vu \in x\backslash \supp(f^{\mathrm{w/o}}) / y$.
\end{proof}

\begin{exa}
\label{exa:hungrystragglers}
Consider a thicket consisting of a single straggler $t$ such that $\al(t) \geq 3$. Since we need at least two summands to perform an expansion, $t$ is the only thicket in its ${\approx}$-equivalence class, but by Proposition \ref{prop:repletepaths} the minimum possible size of a replete subsemigroup containing $t$ is $4$, whence $\supp(t)$ is not replete. This is why we needed to separate out the stragglers above.
% For $u,v \in x\backslash \{t\} / y$, there is no expansion that we can apply to produce $xuvy$ or $xvuy$ as summands, whence $\supp(t)$ is not necessarily a replete subsemigroup. As a concrete counterexample distinct from $t=abac$ of Example \ref{exa:T3notreplete}, let $x = \tr(abca)$, $y=\tr(cadc)$, $u=b$ and $v=d$ (with $a,b,c,d$ generators). Then by straightforward reductions $xuy=xvy$ with tree
% \begin{align*}
% ((((a),b,a,(b)),c,&b,((c),a,c,(a))),d,\\
% &b,(((c),a,c,(a)),d,a,((d),c,d,(c)))),
% \end{align*}
% and this is also the tree of $xvuy$, but $xuvy$ instead has tree:
% \begin{align*}
% ((((a),b,a,(b)),c,&c,((a),b,a,(b))),d,\\
% &b,(((d),c,d,(c)),a,a,((d),c,d,(c)))),
% \end{align*}
% whence if $f$ is the thicket consisting of the single tree $xuy$ then $x \backslash \supp(f) / y$ is not a subsemigroup of $T_n$.
\end{exa}

Abstracting the results above, we arrive at the following definition.

\begin{defn}
\label{defn:complementary}
We say a set $D \subseteq T_n$ of trees is \textbf{sparse} if whenever $s,t \in D$ satisfy $\al(s) \subseteq \al(t)$, we have $s=t$. Given a sparse subset $D$ and a replete subsemigroup $S \subseteq T_n$, we say that $S$ \textbf{dominates} $D$ if the following conditions hold:
\begin{itemize}
	\item $\bar{\al}(D) \cap \bar{\al}(S) = \emptyset$,
	\item $D \cup S$ is a subsemigroup of $T_n$, and
	\item $\al(t)$ is minimal in $\bar{\al}(D \cup S)$ for each $t \in D$.
\end{itemize}

Given a pair $(S,D)$ in which $S$ dominates $D$, a \textbf{parity function} (for $(S,D)$) is a function $p:\Pcal([n]) \to \{0,1\}$ such that $p|_{\bar{\al}(D)} \equiv 1$ and $p(A) = 1$ implies $A \in \bar{\al}(S \cup D)$.

We define a \textbf{complementary triple in $T_n$} to be a triple $(S,D,p)$ consisting of a replete subsemigroup $S \leq T_n$, a sparse subset $D \subseteq T_n$ dominated by $S$ and a parity function $p$ for $(S,D)$.
\end{defn}

\begin{remark}
\label{rem:stability}
Let $f$ be a maximal thicket, define $D := \supp(f^{str})$ and $S := \supp(f^{w/o})$. Given $xuy \in D$ and $xvy \in S$, the evident expansion requires that we should have $xuvy,xvuy \in S$. The conditions of Definition \ref{defn:complementary} are sufficient to guarantee this, thanks to the argument of Lemma \ref{lem:uniformrepleteness}: we observe that $t:= xuyxvy \in D \cup S$ is present by the condition that $D \cup S$ should be a subsemigroup of $T_n$ and work in the replete uniform subsemigroup $S_{\al(t)}$ to deduce that (the left and right branches of) the trees $xuvy,xvuy$ are present. Similar analysis applies to $xuy \neq xvy \in D$. As such, we do not have to impose these as extra restrictions in defining complementary triples.
\end{remark}

\begin{thm}
\label{thm:complements}
There is a bijection between complementary triples $(S,D,p)$ and elements of the free mirig $R_n$. 
\end{thm}
\begin{proof}
In one direction, given an element $x \in R_n$, let $f$ be any maximal thicket representing it. By the preceding results (specifically, Lemma \ref{lem:suppinvariant}, Lemma \ref{lem:fixedparity} and Corollary \ref{cor:unstraggled}) the map $x \mapsto (\supp(f^{\mathrm{str}}),\supp(f^{\mathrm{w/o}}),(A \mapsto P(f_A) \, \mathrm{mod} \, 2))$ is well-defined and satisfies all of the required conditions.

Conversely, given $(S,D,p)$, consider any thicket of the form
\[f:= \sum_{s \in D}s + \sum_{t \in S} k_tt,\]
where the coefficients $k_t \in \{2,3\}$ are chosen such that for each $A \in \al(S)$, $\sum_{t \in S_A} k_t \cong 2 + p(A)$ in $\Nbb_{2,2}$.

\textbf{Claim 1:} Any such $f$ is a maximal thicket.

Suppose that $f$ can be written in the form $f = g + x(u+v)y$. We must show that the expansion $g + x(u+v)y \mapsto g + x(u + v + uv + vu)y$ produces a summand of $f$. By the assumptions on complementary pairs or by Remark \ref{rem:stability} all of the possible choices for $xuy,xvy$ are accounted for to ensure that $xuvy$ and $xvuy$ were already summands of $f$ with coefficients $2$ or $3$, and the expansion above has increased these coefficients by $1$ each. In particular, expanding the same term a second time, we return to $f$, whence $g + x(u + v + uv + vu)y$ is a summand of $f$, as required.

\textbf{Claim 2:} Any pair of thickets constructed in this way are equivalent, so $f$ determines a well-defined element of $R_n$.

Consider $f$ and another thicket of the same form,
\[f' := \sum_{s \in D}s + \sum_{t \in S} k'_tt\]
as forests, in $\Nbb[T_n]$. For each $A \in \bar{\al}(S)$, consider the differences $c_t:= |k_t - k'_t|$ indexed by $t \in S_A$. By the assumption on the coefficients $k_t$ and their sums, $c_t \in \{0,1\}$ for all $t$ and the number of values of $t$ for which $c_t$ is non-zero (the \textit{discrepancy between $f$ and $f'$}) is even. If the discrepancy is zero, $f = f'$, so these are trivially ${\approx}$-equivalent. Otherwise, picking $s \neq t$ with $c_s = c_t = 1$, we consider the expansion $f = g + st + ts \mapsto g + st + ts + s + t$ (using Corollary \ref{cor:product}) which modifies $k_s$ and $k_t$ by $1$, hence ensuring $c_s=c_t=0$ and reducing the discrepancy between $f$ and $f'$ by $2$. Repeating this argument inductively, we expand $f$ into $f'$, as required. This completes the proof.
\end{proof}

\subsection{Triple accounting}

In this section we use all of the work up to this point to compute the cardinality of $R_n$ explicitly for the cases $n=2$ and $n=3$, extending our earlier calculation of the number of replete subsemigroups of $T_2$ and $T_3$.

Let $C_n$ denote the set of complementary triples in $T_n$. Since the number of parity functions for a pair $(S,D)$ depends only on $S$, we can decompose $|R_n| = |C_n|$ into a sum indexed by replete subsemigroups. A naive count results in the following formula:
\begin{equation}
\label{eq:Rn1}
|R_n| = \sum_{\text{replete }S \subseteq T_n} \#\{D \mid S \text{ dominates } D\} 2^{|\bar{\al}(S)|};
\end{equation}
While it is possible to use this formula directly to compute $|R_2|$, an alternative counting strategy can be derived by considering the replete subsemigroup generated by $D \cup S$ when $S$ dominates $D$. This will be a subsemigroup in which there is exactly one leftmost and rightmost path in the uniform subsemigroups corresponding to alphabets of elements of $D$.

\begin{prop}
\label{prop:Rn2}
For a replete subsemigroup $S \leq T_n$, let $\bar{\al}_m(S)$ be the subset of $\bar{\al}(S)$ consisting of the minimal elements $A$ for which $|\lmp(S_A)| = |\rmp(S_A)| = 1$. Then we have:
\begin{equation}
\label{eq:Rn2}
|R_n| = \sum_{\text{replete }S \leq T_n} \sum_{E \subseteq \bar{\alpha}_m(S)} 2^{\left(|\bar{\al}(S)|-|E|\right)} \prod_{A \in E} \prod_{j=1}^{|A|-1} j^{2^{(|A|-j+1)}-2}.
\end{equation}
\end{prop}
\begin{proof}
The inner sum, indexed by $\bar{\alpha}_m(S')$, counts the possible pairs $(S,D)$ such that $S$ dominates $D$ and $\langle S \cup D \rangle = S'$, grouped according to $E = \al(D)$. The power of $2$ counts the number of partition functions, while the product is the number of possible choices of $D$, using (the square of) the formula of Lemma \ref{lem:repletebranchset}.
\end{proof}

To reorganize this sum into a form compatible with our earlier counts of replete subsemigroups of $T_2$ and $T_3$, consider a replete subsemigroup $S'$ containing the trivial tree. Here $\bar{\alpha}_m(S') = \{\emptyset\}$, and we extract the summand,
\[2^{|\bar{\al}(S)|} + 2^{\left(|\bar{\al}(S)|-1\right)} = 3 \cdot 2^{\left(|\bar{\al}(S)|-1\right)}.\]
As such, we may rewrite \eqref{eq:Rn2} as a sum over replete subsemigroups $S'$ not containing the trivial tree (adding $1$ to the exponents from the previous expression):
\begin{equation}
\label{eq:Rn3}
\sum_{\substack{\text{replete }S \leq T_n \\ \emptyset \notin \bar{\al}(S)}} \left[ 3 \cdot 2^{|\bar{\al}(S)|} + \sum_{E \subseteq \bar{\alpha}_m(S)} 2^{\left(|\bar{\al}(S)|-|E|\right)} \prod_{A \in E} \prod_{j=1}^{|A|-1} j^{2^{(|A|-j+1)}-2} \right].
\end{equation}
Conveniently, the product appearing in this expression,
\[Q_E:= \prod_{A \in E} \prod_{j=1}^{|A|-1} j^{2^{(|A|-j+1)}-2},\]
always evaluates to $1$ when the number of generators $n \leq 3$ but for the exceptional case $[3] \in E$, in which case it evaluates to $4$. This can only occur when $S$ is a minimal uniform replete semigroup of height $3$.

\begin{exa}
\label{exa:R2count}
By Lemma \ref{lem:T2replete}, the $42$ subsemigroups of $T_2$ are all replete. Combined with the observations above, it follows that for $n =2$, \eqref{eq:Rn3} simplifies to the following formula:
\[|R_2|= \sum_{\substack{S \leq T_2 \\ 1 \notin S}} \left[ 3 \cdot 2^{|\bar{\al}(S)|} + \sum_{E \subseteq \bar{\alpha}_m(S)} 2^{\left(|\bar{\al}(S)|-|E|\right)} \right].\]
Tabulating the count of these subsemigroups as in Example \ref{exa:T2subsemi} we arrive at Table \ref{table:R2count0}.
\begin{table}[ht]
\centering
\DivWidth{\colOne}{$\{\{a,b\}\}$}{2}
\begin{tabu}{ l |c|c|C{\colOne}|C{\colOne}|c|c|}
$\bar{\al}(S)$ & $\emptyset$ & $\{\{a\}\}$ & \multicolumn{2}{c|}{$\{\{a,b\}\}$} & $\{\{a\},\{a,b\}\}$ & $\{\{a\},\{b\},\{a,b\}\}$\\
\hline
Factor     & $1$ & $2$ & \multicolumn{2}{c|}{$1$} & $2$ & $1$ \\
\hline
Replete $S$  & $1$ & $1$ & $4$ & $5$ & $4$ & $1$ \\
$3 \cdot 2^{|\bar{\al}(S)|}$ & $3$ & $6$ & $6$ & $6$ & $12$ & $24$\\
$\sum 2^{\left(|\bar{\al}(S)|-|E|\right)}$ & $1$ & $3$ & $3$ & $2$ & $6$ & $18$ \\
\hline
Contribution & $4$ & $18$ & \multicolumn{2}{c|}{$76$} & $144$ & $42$
\end{tabu}
\caption{Counting the number of elements of $R_2$ using the simplified version of \eqref{eq:Rn3}. Note that in the third column we must distinguish the $4$ cases where $S$ is a singleton from the remaining cases where $S$ contains two or more elements for the definition of $\bar{\al}_m(S)$.}
\label{table:R2count0}
\end{table}
Summing the entries in the final row, we recover
\[|R_2| = 284.\]
\end{exa}

The case splitting in Example \ref{exa:R2count} is unavoidable. However, having enumerated all of the replete subsemigroups of $T_3$ explicitly (via their branch sets, in Table \ref{table:T3count}), it suffices to group the subsemigroups for each alphabet by the size of $\bar{\al}_m(S)$. This is the main challenge impjicit in the calculations of the following example.

% Earlier version of the $R_2$ calculations:

%There are $9$ sparse sets in $T_2$: the empty set, the singletons and $\{(a),(b)\}$. It remains to check which subsemigroups dominate which sparse sets.
%
%For subsemigroups $S$ containing $1$, the only compatible sparse subset $D$ (in the sense that $\al(s) \not\subseteq \al(t)$ for all $s \in S$) is $D = \emptyset$.
%
%If we do not have $1 \in S$, we may have $D = \{1\}$, which is automatically compatible. To have $D = \{a\}$ (resp. $D = \{b\}$) we either require $\bar{\al}(S) = \emptyset$, or $\bar{\al}(S) = \{\{a,b\}\}$ and $S$ closed under multiplication by $a$ (resp. $b$), or $b \in S$ (resp. $a \in S$) and $S$ contains \textit{all} trees $t$ with $\al(t) = \{a,b\}$, since we require $S \cup D$ to be a subsemigroup (one can check directly that this is sufficient for $S$ to dominate $D$ in this case). The four choices of $D = \{t\}$ where $\al(t) = \{a,b\}$ require $S$ to be empty.
%
%Finally, the only sparse subset containing more than one element is $D = \{a,b\}$, which forces all products of these to be present in $S$, and these are the only elements allowed.

\begin{exa}
\label{exa:R3count}
To compute $|R_3|$, we exploit the additional detail included in Table \ref{table:T3count}, keeping the above simplifications of the calculation in mind. We record the calculation in Table \ref{table:R3count}. The hardest calculation involved in generating this table is the case $\bar{\al}(S) = \{\{a,b\},\{b,c\},\{a,c\},\{a,b,c\}\}$, where we must separate the cases where none, one, two or three of the minimal elements are members of $\bar{\al}_m(S)$. Summing the contributions from the different alphabets, we obtain:
\[ |R_3| = 510605. \]
\end{exa}

\begin{remark}
\label{rem:hardcount}
Our strategy in this section has been to exploit our existing count of replete semigroups from earlier, since this seemed simpler than determining the number of replete subsemigroups which dominate a given sparse set. We expect further study of the algebras of rightmost paths should reveal a more systematic understanding of this relationship.

However, even if this is possible, there is good reason to believe that combinatorial tools capable of precisely enumerating the elements of $|R_n|$ for $n$ larger than $8$ or $9$ do not yet exist (even with extensive computational power). Indeed, just counting the sparse sets should be at least as challenging as counting their alphabets, and these are (by definition) antichains in $\Pcal([n])$. The problem of finding a closed form expression for the number of such antichains is exactly \textit{Dedekind's problem}, which has been open for over $125$ years, Kisielewicz' arithmetic presentation of the problem \cite{Kisielewicz} as a sum in which every subset of $\Pcal([n])$ provides a summand evaluating to $0$ or $1$ notwithstanding. The $n = 9$ case was only computed with significant computational resources in 2023, independently by J\"{a}kel \cite{Jakel} and Van Hirtum \textit{et al} \cite{SuperDedekind}.

Nonetheless, there is scope to lift the methods used by Kleitman and Markowsky to obtain asymptotic bounds on Dedekind numbers \cite{Dedekind} to improve the rudimentary bounds of Lemma \ref{lem:Rnfinite} and Example \ref{exa:R1}.
\end{remark}

\begin{landscape}
\begin{table}[ht]
\centering
\begin{tabu}{ l |c|c|c|c|c|c|}
$\bar{\al}(S)$ & $\emptyset$ & $\{\{a\}\}$ & \multicolumn{2}{c|}{$\{\{a,b\}\}$} & $\{\{a\},\{a,b\}\}$ & $\{\{a\},\{b\},\{a,b\}\}$\\
\hline
Factor     & $1$ & $3$ & \multicolumn{2}{c|}{$3$} & $6$ & $3$ \\
\hline
Replete $S$  & $1$ & $1$ & $4$ & $5$ & $4$ & $1$ \\
$3 \cdot 2^{|\bar{\al}(S)|}$ & $3$ & $6$ & $6$ & $6$ & $12$ & $24$\\
$\sum 2^{\left(|\bar{\al}(S)|-|E|\right)}$ & $1$ & $3$ & $3$ & $2$ & $6$ & $18$ \\
\hline
Contribution & $4$ & $27$ & $108$ & $120$ & $432$ & $126$
\end{tabu}
\\
\vspace{6pt}
\DivWidth{\colOne}{$12\{\{a,b\}\}12$}{2}
\begin{tabu}{ l |c|c|c|c|c|c|c|c|c|c|c|c|c|}
$\bar{\al}(S) \backslash \{[3]\}$
& \multicolumn{2}{c|}{$\emptyset$} & $\{\{a\}\}$ & \multicolumn{2}{c|}{$\{\{a,b\}\}$}
& $\{\{a\},\{a,b\}\}$ & \multicolumn{2}{c|}{$\{\{a\},\{b,c\}\}$} & \multicolumn{3}{c|}{$\{\{a,b\},\{b,c\}\}$}
& $\{\{a\},\{b\},\{a,b\}\}$
\\
\hline
Factor
& \multicolumn{2}{c|}{$1$}         & $3$         & \multicolumn{2}{c|}{$3$}
& $6$                 & \multicolumn{2}{c|}{$3$}                 & \multicolumn{3}{c|}{$3$}
& $3$ \\
\hline
Replete $S$
& $36$  & $448$                     & $144$      & $324$   & $205$
& $196$               & $36$   & $13$                              & $196$   & $256$   & $77$
& $25$ \\
$3 \cdot 2^{|\bar{\al}(S)|}$
& $6$   & $6$                       & $12$       & $12$    & $12$
& $24$                & $24$   & $24$                              & $24$    & $24$    & $24$
& $48$  \\
$\sum 2^{\left(|\bar{\al}(S)|-|E|\right)}Q_E$
& $10$  & $2$                       & $6$        & $6$     & $4$ 
& $12$                & $18$   & $12$                              & $18$    & $12$    & $8$
& $36$ \\
\hline
Contribution
& $576$ & $3584$                    & $7776$     & $17496$ & $9840$
& $42336$             & $4536$ & $1404$                              & $24696$ & $27648$ & $7392$
& $6300$
\end{tabu}
\\
\vspace{6pt}
\DivWidth{\colOne}{$\{\{a\},\{a,b\},\{b,c\}\}$}{2}
%\DivWidth{\colTwo}{$\{\{b\},\{a,b\},\{b,c\}\}$}{3}
\begin{tabu}{ l |C{\colOne}|C{\colOne}|c|c|c|c|c|c|}
$\bar{\al}(S) \backslash \{[3]\}$ & \multicolumn{2}{c|}{$\{\{a\},\{a,b\},\{b,c\}\}$} & $\{\{b\},\{a,b\},\{b,c\}\}$ & $\{\{a\},\{b\},\{a,b\},\{b,c\}\}$ & \multicolumn{4}{c|}{$\{\{a,b\},\{b,c\},\{a,c\}\}$} \\
\hline
Factor          & \multicolumn{2}{c|}{$6$}
                & $3$  & $6$ & \multicolumn{4}{c|}{$1$}\\
\hline
Replete $S$ & $81$  & $40$   & $169$ & $16$  & $196$           & $495$             & $450$           & $155$ \\
$3 \cdot 2^{|\bar{\al}(S)|}$ & $48$ & $48$ & $48$ & $96$ & $48$ & $48$ & $48$ & $48$ \\
$\sum 2^{\left(|\bar{\al}(S)|-|E|\right)}$ & $36$ & $24$ & $24$ & $72$ & $54$ & $36$ & $24$ & $16$ \\
\hline
Contribution & $40824$ & $17280$ & $36504$ & $16128$ & $19992$ & $41580$ & $32400$ & $9920$
\end{tabu}
\\
\vspace{6pt}
\DivWidth{\colOne}{$\{\{a\},\{a,b\},\{b,c\},\{a,c\}\}$}{2}
\begin{tabu}{ l |C{\colOne}|C{\colOne}|c|c|}
$\bar{\al}(S) \backslash \{[3]\}$ & \multicolumn{2}{c|}{$\{\{a\},\{a,b\},\{b,c\},\{a,c\}\}$} & $\{\{a\},\{b\},\{a,b\},\{b,c\},\{a,c\}\}$ & all \\
\hline
Factor          & \multicolumn{2}{c|}{$3$} & $3$ & $1$ \\
\hline
Replete $S$ & $144$           & $112$    & $25$     & $1$  \\
$3 \cdot 2^{|\bar{\al}(S)|}$ & $96$ & $96$ & $192$ & $384$ \\
$\sum 2^{\left(|\bar{\al}(S)|-|E|\right)}$ & $72$ & $48$ & $72$ & $432$ \\
\hline
Contribution & $72576$ & $48384$ & $19800$ & $816$
\end{tabu}
\caption{An accounting of the complementary triples $(S,D,p)$ tabulated according to the replete subsemigroup $\langle S \cup D \rangle_r$ they generate. The first table accounts for the cases where $[3] \notin \bar{\al}(S)$, while the remaining tables cover $[3] \in \bar{\al}(S)$. As before, the `Factor' row accounts for the alphabets obtained by permuting the generators. The third row of each table is generated from Table \ref{table:T3count}, with columns subdivided according to the corresponding value of $\bar{\al}_m(S)$. In the second table, the product $Q_E$ is $4$ when $E = \{[3]\}$ and is $1$ otherwise, so we omit this factor from the other tables.}
\label{table:R3count}
\end{table}
\end{landscape}

\clearpage

\subsection{The mirig of complementary triples}
\label{ssec:ops}

Just as in Section \ref{ssec:trees} the bijection of Theorem \ref{thm:complements} induces a mirig structure on $C_n$. For completeness, we would like to make the translated rig operations explicit, but this requires some extra work, since we have not yet examined how these operations are realised at the level of maximal thickets.

Observe that if $f$ is a maximal thicket and $g \approx f$, then equipping $\Nbb_{2,2}$ with the preorder inherited from $\Nbb$ we have the apparity inequality $P(f_A) \geq P(g_A)$ for each $A$. In fact we have the stronger properties that that $P(g_A) = 0$ implies $P(f_A) = 0$ or $2$; $P(g_A) = 2$ implies $P(f_A) =2$, and so on. In particular, each expansion of $g$ adds $2$ to the parity of some $g_A$, so an expansion cannot create new stragglers. As such, the stragglers of $f$ are contained in the stragglers of $g$.

\begin{lem}
\label{lem:prodstragglers}
Let $f,f',f''$ be maximal thickets such that $f'' \approx ff'$. Then ${f''}^{\mathrm{str}}$ consists of those $tt' \in {f}^{\mathrm{str}}{f'}^{\mathrm{str}}$ such that for all $s \in f$, $s' \in f'$, $\al(ss') \subseteq \al(tt')$ implies $s=t$ and $s'=t'$.
\end{lem}
\begin{proof}
By the preceding observation, the stragglers of $f''$ are amongst the stragglers of
\[ff' = ({f}^{\mathrm{str}} + {f}^{\mathrm{w/o}})({f'}^{\mathrm{str}} + {f'}^{\mathrm{w/o}}),\]
which in turn are contained in the first factor, ${f}^{\mathrm{str}}{f'}^{\mathrm{str}}$, since all other summands have coefficients $2$ or $3$ already. The only way that a tree $tt'$ can fail to be a straggler of the product is if there is another summand of $ff'$ (formed via a distinct product) whose alphabet is contained in that of $tt'$, whence the given condition is necessary. Conversely, if $tt'$ is lonely and $\al(tt')$ is minimal in $\bar{\al}(\supp(ff'))$ then any expansion of $ff'$ necessarily leaves $tt'$ lonely, so $tt'$ will be a straggler in $f''$, as required.
\end{proof}

\begin{lem}
\label{lem:sumstragglers}
Let $f,f',f''$ be maximal thickets such that $f'' \approx f+f'$. Then ${f''}^{\mathrm{str}}$ consists of those $t \in {f}^{\mathrm{str}}$ such that for all $t' \in f'$, $\al(t') \not\subseteq \al(t)$, alongside those $t' \in {f'}^{\mathrm{str}}$ such that for all $t \in f$, $\al(t) \not\subseteq \al(t')$.
\end{lem}
\begin{proof}
Once again, the stragglers of $f''$ are amongst the stragglers of $f+f'$, which in turn are contained in the union of the stragglers of $f$ and those of $f'$. The only way that a straggler $t$ of $f$ can fail to be a straggler of the sum is if there is another summand of $f+f'$ whose alphabet is contained in that of $t$, which must necessarily come from $f'$, whence the given conditions are necessary; they are sufficient by an analogous argument to that for Lemma \ref{lem:prodstragglers}.
\end{proof}

For both the product and the sum we can only recover the replete subsemigroups $\supp({f''}^{\mathrm{w/o}})$ by closing under the requirements of Definition \ref{defn:complementary}.

\begin{prop}
\label{prop:dompair}
Let $D,U \subseteq T_n$ be disjoint subsets such that $D$ is sparse and $\al(s) \not\subseteq \al(t)$ whenever $s \in U$ and $t \in D$. Then the replete subsemigroup $\langle U \rangle_r$ generated by $U$ (see Corollary \ref{cor:repleteintersect}) dominates $D$.
\end{prop}
\begin{proof}
Clearly the subsemigroup generated by $U$ also satisfies the condition on the alphabets, and by Remark \ref{rem:repletealphabets} generating the replete subsemigroup $\langle U \rangle_r$ doesn't create trees with new alphabets, whence it dominates $D$, as required.
\end{proof}

\begin{cor}
\label{cor:domops}
Let $f,f',f''$ be maximal thickets such that $f'' \approx ff'$, let $D := \supp({f''}^{\mathrm{str}})$ as computed in Lemma \ref{lem:prodstragglers} and let $U := \supp(ff') \backslash D$. Then $\supp({f''}^{\mathrm{w/o}}) = \langle U \rangle_r$. If instead $f'' \approx f+f'$, $D$ is the subset computed in Lemma \ref{lem:prodstragglers} and $U := \supp(f+f') \backslash D$ then the corresponding identity holds.
\end{cor}
\begin{proof}
Definition \ref{defn:complementary} is such that any expansion of $ff'$ (resp. of $f+f'$) produces elements which are necessarily contained in any replete subsemigroup $S$ containing $U$ and dominating $D$. Since $f''$ is obtained by expansions, $\supp(f'')$ must coincide with the smallest such subsemigroup. 
\end{proof}

Finally, we need to determine the parity function for the sum and product. This is straightforward. It is again convenient to observe that given $f,f'$ maximal, the parity assigned to uniform subsemigroups in the support of $f+f'$ only increases by $0$ or $2$ with each expansion, from which we deduce the following.

\begin{lem}
Let $f,f',f''$ be maximal thickets such that $f'' \approx ff'$ (resp. $f'' \approx f+f'$) and let $S = \supp({f''}^{\mathrm{w/o}})$, as calculated in Corollary \ref{cor:domops}. Then for $A \in \bar{\al}(S)$, $P(f''_A) = 3$ if and only if $P((ff')_A) = 1$ or $3$, and $P(f''_A)=2$ otherwise. The analogous statement holds when $f'' \approx f+f'$.
\end{lem}

We summarize these results in the language of complementary triples.

\begin{thm}
\label{thm:ops}
Let $(S,D,p),(S',D',p') \in C_n$ be complementary triples. Let
\begin{align*}
B &:= \{tt' \in DD' \mid \forall s \in S \cup D, \, \forall s' \in S' \cup D', \, \al(ss') \subseteq \al(tt') \Rightarrow s=t \text{ and } s'=t'\}, \\
E &:= \{t \in D \mid \forall t' \in S' \cup D', \, \al(t') \not\subseteq \al(t)\} \cup \{t' \in D' \mid \forall t \in S \cup D, \, \al(t) \not\subseteq \al(t')\}.
\end{align*}
Then the product operation on $C_n$ inherited from $R_n$ satisfies:
\[(S,D,p) \cdot (S',D',p') = (\langle (S \cup D)\cdot(S' \cup D') \backslash B\rangle_r, B, pp'),\] where $pp'(A) = \sum_{A_1 \cup A_2 = A}p(A_1)p'(A_2)$, reduced modulo $2$. Meanwhile, the sum on $C_n$ is given by:
\[(S,D,p) + (S',D',p') = (\langle (S \cup D) \cup (S' \cup D') \backslash E \rangle_r, E, p+p'),\]
where $p+p'(A):= p(A) + p'(A)$, reduced modulo $2$.

Finally, $0 \in R_n$ corresponds to $(\emptyset,\emptyset,0)$ and $1 \in R_n$ corresponds to $(\emptyset,\{1\},\chi_{\emptyset})$, where $\chi_{\emptyset}$ is the characteristic function for the empty set.
\end{thm}

We note that the convenient expressions for the parity functions motivated our choice of presentation for complementary triples.

\section{More mirigs}
\label{sec:examples}

To conclude this article, we return to considering mirigs beyond free mirigs.

\subsection{Various characteristics}

\begin{exa}
As mentioned in Remark \ref{rem:terminology}, (additively) \textit{idempotent semirings} refer to rigs in which $x+x=x$ for all elements $x$. This in particular forces $1+1 = 1$, so these are precisely the rigs having characteristic $(1,1)$ (or $(0,1)$). Additively idempotent mirigs include \textbf{distributive lattices} with top and bottom elements $1,0$ and $+,\cdot$ taken to be $\vee,\wedge$ respectively: both operations are idempotent and $\wedge$ distributes over $\vee$.
\end{exa}

% This is not the only example of additively idempotent mirig. Given a rig $R$, we call an non-unital subrig of $R$ (that is, a subset $S \subseteq R$ containing $0$ and closed under addition and products) a \textbf{standard}. Note that we avoid recycling the word `ideal' because some rigs are rings, and the two concepts do not coincide in general.

% \begin{lem}
% \label{lem:standard}
% Given a mirig $R$, let $\Scal(R)$ be the set of standards of $R$. Then $\Scal(R)$ is an additively idempotent mirig with addition and multiplication defined pointwise:
% \begin{align*}
% S_1 S_2 &:= \{\sum_{i}^{n} \prod_{j=1}^{m} s_{1,i,j}s_{2,i,j} \mid s_{1,i,j} \in S_1, s_{2,i,j} \in S_2, n,m \in \Nbb\}\\
% S_1 + S_2 &:= \{s_1 + s_2 \mid s_1 \in S_1, s_2 \in S_2\}
% \end{align*}
% Zero is the subset $\{0\}$ and one is $R$.
% \end{lem}
% \begin{proof}
% Distributivity of multiplication over addition ensures that the sum of standards is closed  The associativity of addition and multiplication and commutativity of addition follow from them holding elementwise. Distributivity is similarly straightforward. Idempotency of addition holds thanks to standards containing $0$ while idempotency of multiplication relies on elements of $R$ being idempotent.
% \end{proof}

% Observe that the $\Scal$ construction is functorial, in the sense that a rig homomorphism $h:R \to R'$ induces a rig homomorphism $\Scal(R) \to \Scal(R')$ by applying $h$ pointwise on standards and then closing under sums.

\begin{exa}
\label{exa:(1,1)}
Consider the quotient of a free mirig $R_n$ by $1 \sim 1+1$. Examining Theorem \ref{thm:ops}, we see that this corresponds to identifying a complementary triple $(S,D,p)$ with the triple $(\langle S \cup D \rangle_r,\emptyset,0)$. In other words, the elements of the resulting mirig are in bijection with the replete subsemigroups of $T_n$, with operations
\[S \cdot S' = \langle \{ss' \mid s \in S, s' \in S'\} \rangle_r\]
\[S + S' = \langle S \cup S' \rangle_r. \]
Observe that the multiplication is not commutative, since for distinct generators $a,b$, the replete subsemigroups $\{a\},\{b\}$ have distinct products $\{ab\}$ and $\{ba\}$ depending on the order of multiplication. Thus a rig (or mirig) of characteristic $(1,1)$ need not be commutative, although other known structures (such as non-commutative \textit{quantales}) already witness this fact.
\end{exa}

This example raises the question of what can be said for the other characteristics. We of course discard the degenerate case $(0,1)$. For brevity, we write $R^{(p,q)}_n$ for the free mirig of characteristic $(p,q)$ (that is, $R_n$ quotiented by the congruence generated by $p+q \sim p$), where $(p,q)$ varies amongst the characteristics from Lemma \ref{lem:char}. From the above example and Section \ref{ssec:replete}, we already have the following cardinalities:
\begin{equation}
	|R^{(1,1)}_n| = 2, 4, 42, 18030, \dotsc
\end{equation}

\begin{exa}
\label{exa:(2,1)}
Intermediate between free mirigs and mirigs satisfying $x+x=x$ are those of characteristic $(2,1)$ where $3x = 2x$. Considering maximal thickets representing elements of $R_n$, we see that this identity flattens the distinction between the even and odd coefficients of any non-stragglers. In terms of complementary triples, this amounts to dropping parity functions, so the free mirig on $[n]$ subject to the identity $2 \sim 3$ can be identified with the collection of pairs $(S,D)$ where $S \leq T_n$ is a replete subsemigroup and $D \subseteq T_n$ is a sparse subset dominated by $S$.

To count these, we adapt Proposition \ref{prop:Rn2} and the subsequent argument, removing the term counting the number of parity functions, to obtain:
\begin{equation}
\label{eq:R(2,1)n}
|R^{(2,1)}_n| = \sum_{\substack{\text{replete }S \leq T_n \\ \emptyset \notin \bar{\al}(S)}} \left[ 2 + \sum_{E \subseteq \bar{\alpha}_m(S)} \prod_{A \in E} \prod_{j=1}^{|A|-1} j^{2^{(|A|-j+1)}-2} \right],
\end{equation}
where the first term counts when $1 \in \langle S \cup D \rangle_r$, so either $1 \in S$ and $D = \emptyset$ or $1 \notin S$ and $D = \{1\}$, and the second term counts when $1 \notin \langle S \cup D \rangle_r$. For $n \leq 3$ the product appearing in this expressiong (which we again abbreviate to $Q_E$) reduces to $1$ except in the exceptional case that $\bar{\al}_m(S) = [3]$, when it evaluates to $4$. Thus for all but the exceptional case, the inner sum reduces to a count of the subsets of $\bar{\al}_m$.

For $n=0$ we recover $\Nbb_{2,1}$. For $n=1$, we can more directly count by extending Example \ref{exa:R1}: of the $9$ elements of $\Nbb_{2,1}[M_1]$, $x^2 = x$ identifies $1+a$ with $1+2a$ and $2+a$ with $2+2a$, leaving $7$ distinct elements of $R^{(2,1)}_1$.

For $R^{(2,1)}_2$ and $R^{(2,1)}_3$, we proceed as we did for $R_2$ and $R_3$, respectively, constructing a table to compute the sum by alphabets. See Tables \ref{table:R(2,1)2count} and \ref{table:R(2,1)count}.
\begin{table}[ht]
\centering
\DivWidth{\colOne}{$\{\{a,b\}\}$}{2}
\begin{tabu}{ l |c|c|C{\colOne}|C{\colOne}|c|c|}
$\bar{\al}(S)$ & $\emptyset$ & $\{\{a\}\}$ & \multicolumn{2}{c|}{$\{\{a,b\}\}$} & $\{\{a\},\{a,b\}\}$ & $\{\{a\},\{b\},\{a,b\}\}$\\
\hline
Factor     & $1$ & $2$ & \multicolumn{2}{c|}{$1$} & $2$ & $1$ \\
\hline
Replete $S$  & $1$ & $1$ & $4$ & $5$ & $4$ & $1$ \\
$2 + \sum Q_E$ & $3$ & $4$ & $4$ & $3$ & $4$ & $6$ \\
\hline
Contribution & $3$ & $8$ & \multicolumn{2}{c|}{$31$} & $32$ & $6$
\end{tabu}
\caption{An accounting of pairs $(S,D)$ with $D \subseteq T_2$ sparse and $S \leq T_2$ dominating $D$, tabulated analogously to Table \ref{table:R2count0}.}
\label{table:R(2,1)2count}
\end{table}
In summary, we find that:
\begin{equation}
\label{eq:R(2,1)count}
|R^{(2,1)}_n| = 3,7,80,40601,\dotsc
\end{equation}
\end{exa}

\begin{exa}
\label{exa:(1,2)}
Forcing the characteristic to be $(1,2)$ (so imposing $3x = x$) instead enables us to treat lonely trees as if they were duplicated (or rather, triplicated). As such, a maximal thicket in $\Nbb_{1,2}[T_n]$ representing an element of $R^{(1,2)}_n$ must have a replete subsemigroup of summands. On the other hand, parity functions are respected. We conclude that elements of $R^{(1,2)}_n$ are in bijection with the collection of pairs $(S,p)$ where $S \leq T_n$ is a replete subsemigroup and $p: \Pcal([n]) \to \Nbb_{0,2}$ is a parity function whose value is $0$ on the complement of $\bar{\al}(S)$. Separating the cases with and without the identity, we have:
\begin{equation}
\label{eq:R(1,2)}
|R^{(1,2)}_n| = 3\sum_{\substack{\text{replete }S \leq T_n \\ \emptyset \notin \bar{\al}(S)}} 2^{|\bar{\al}(S)|},
\end{equation}
where we move the coefficient of $3$ outside the sum to deduce that $|R^{(1,2)}_n|$ is always a multiple of $3$. The story is simple for $n=0$ and $1$, recovering $\Nbb_{1,2}$ and $\Nbb_{1,2}[T_n]$, respectively (in the latter case, no further identifications are made by imposing $x^2 = x$). For $n = 2$ and $3$, we can directly reuse the counts in Tables \ref{table:R2count0} and \ref{table:R3count}. In so doing, we obtain:
\begin{equation}
|R^{(1,2)}_n| = 3,9,189,160389,\dotsc
\end{equation}
\end{exa}

Implicit in Examples \ref{exa:(1,1)}, \ref{exa:(2,1)} and \ref{exa:(1,2)} is the fact that imposing the characteristic identity leaves the set of alphabets of a maximal thicket invariant, since equivalent thickets can only be produced from (or absorbed into) trees which are summands of that thicket. As such, the arguments of Section \ref{ssec:forests} and \ref{ssec:ops} can be reapplied to arrive at the given conclusions. For the final characteristic $(0,2)$ this is no longer the case, since for any forest $f$ and tree $t$, $f \simeq' f+t+t$. However, we have already seen in Section \ref{sec:Boole} that in this case we are working with Boolean rings, which are commutative, so all trees on the same alphabet become equivalent. Since the parity function of a complementary triple is respected by the reduction to characteristic $(2,0)$, we arrive at the following simple presentation.

\begin{exa}
\label{exa:(0,2)}
Imposing the equation $2x = 0$ reduces $R_n$ to the free Boolean ring on $n$ generators, $\Pcal(\Pcal([n]))$.\footnote{Beware that for infinite sets the powerset functor should be replaced for the `finite powerset' functor returning the set of finite subsets.} The quotient map sends a complementary triple $(S,D,p)$ to $p$, viewed as the characteristic function of a subset of $\Pcal([n])$. For completeness, we record:
\begin{equation}
\label{eq:R(2,0)}
|R^{(2,0)}_n| = 2,4,16,256,\dotsc
\end{equation}
\end{exa}

\begin{landscape}
\begin{table}[ht]
\centering
\begin{tabu}{ l |c|c|c|c|c|c|}
$\bar{\al}(S)$ & $\emptyset$ & $\{\{a\}\}$ & \multicolumn{2}{c|}{$\{\{a,b\}\}$} & $\{\{a\},\{a,b\}\}$ & $\{\{a\},\{b\},\{a,b\}\}$\\
\hline
Factor & $1$ & $3$ & \multicolumn{2}{c|}{$3$} & $6$ & $3$ \\
\hline
Replete $S$  & $1$ & $1$ & $4$ & $5$ & $4$ & $1$ \\
$2 + \sum Q_E$ & $3$ & $4$ & $4$ & $3$ & $4$ & $6$\\
\hline
Contribution & $3$ & $12$ & \multicolumn{2}{c|}{$93$} & $96$ & $18$
\end{tabu}
\\
\vspace{6pt}
\DivWidth{\colOne}{$12\{\{a,b\}\}12$}{2}
\begin{tabu}{ l |c|c|c|c|c|c|c|c|c|c|c|c|c|}
$\bar{\al}(S) \backslash \{[3]\}$ & \multicolumn{2}{c|}{$\emptyset$} & $\{\{a\}\}$ & \multicolumn{2}{c|}{$\{\{a,b\}\}$} & $\{\{a\},\{a,b\}\}$ & \multicolumn{2}{c|}{$\{\{a\},\{b,c\}\}$} & \multicolumn{3}{c|}{$\{\{a,b\},\{b,c\}\}$} & $\{\{a\},\{b\},\{a,b\}\}$
\\
\hline
Factor          & \multicolumn{2}{c|}{$1$}  & $3$  & \multicolumn{2}{c|}{$3$} & $6$ & \multicolumn{2}{c|}{$3$} & \multicolumn{3}{c|}{$3$} & $3$ \\
\hline
Replete $S$ & $36$ & $448$ & $144$ & $324$ & $205$ & $196$ & $36$ & $13$ & $196$           & $256$             & $77$  & $25$ \\
$2 + \sum Q_E$ & $7$ & $3$ & $4$ & $4$ & $3$ & $4$ & $6$ & $4$ & $6$ & $4$ & $3$ & $6$ \\
\hline
Contribution & \multicolumn{2}{c|}{$1596$} & $1728$ & \multicolumn{2}{c|}{$5733$} & $4704$ & \multicolumn{2}{c|}{$804$}& \multicolumn{3}{c|}{$7293$} & $450$
\end{tabu}
\\
\vspace{6pt}
\DivWidth{\colOne}{$\{\{a\},\{a,b\},\{b,c\}\}$}{2}
%\DivWidth{\colTwo}{$\{\{b\},\{a,b\},\{b,c\}\}$}{3}
\begin{tabu}{ l |C{\colOne}|C{\colOne}|c|c|c|c|c|c|}
$\bar{\al}(S) \backslash \{[3]\}$ & \multicolumn{2}{c|}{$\{\{a\},\{a,b\},\{b,c\}\}$} & $\{\{b\},\{a,b\},\{b,c\}\}$ & $\{\{a\},\{b\},\{a,b\},\{b,c\}\}$ & \multicolumn{4}{c|}{$\{\{a,b\},\{b,c\},\{a,c\}\}$} \\
\hline
Factor          & \multicolumn{2}{c|}{$6$}
                & $3$  & $6$ & \multicolumn{4}{c|}{$1$}\\
\hline
Replete $S$ & $81$  & $40$   & $169$ & $16$  & $196$           & $495$             & $450$           & $155$ \\
$2 + \sum Q_E$ & $6$ & $4$ & $4$ & $6$ & $10$ & $6$ & $4$ & $3$ \\
\hline
Contribution & \multicolumn{2}{c|}{$3876$} & $2028$ & $576$ & \multicolumn{4}{c|}{$7195$}
\end{tabu}
\\
\vspace{6pt}
\DivWidth{\colOne}{$\{\{a\},\{a,b\},\{b,c\},\{a,c\}\}$}{2}
\begin{tabu}{ l |C{\colOne}|C{\colOne}|c|c|}
$\bar{\al}(S) \backslash \{[3]\}$ & \multicolumn{2}{c|}{$\{\{a\},\{a,b\},\{b,c\},\{a,c\}\}$} & $\{\{a\},\{b\},\{a,b\},\{b,c\},\{a,c\}\}$ & all \\
\hline
Factor          & \multicolumn{2}{c|}{$3$} & $3$ & $1$ \\
\hline
Replete $S$ & $144$           & $112$    & $25$     & $1$  \\
$2 + \sum Q_E$ & $6$ & $4$ & $6$ & $10$ \\
\hline
Contribution & \multicolumn{2}{c|}{$3936$} & $450$ & $10$
\end{tabu}
\caption{An accounting of pairs $(S,D)$ with $D \subseteq T_3$ sparse and $S \leq T_3$ replete dominating $D$, tabulated analogously to Table \ref{table:R3count}.}
\label{table:R(2,1)count}
\end{table}
\end{landscape}

\subsection{Returning to Boolean semirings}

\begin{exa}
Related to Example \ref{exa:(1,2)} are Guzm\'{a}n's Boolean semirings, where we impose $1 + x+x = 1$ (which implies $1+1+1 = 1$). Representing elements of $R^{(1,2)}_n$ as pairs $(S,p)$, we see that $(S,p)$ is related to $(S',p)$, where $S'$ is the maximal subsemigroup over $S$, meaning:
\[S'_{A'} =
\begin{cases}
(T_n)_{A'} & \text{ if }\exists A \in \bar{\al}(S), \, A \subseteq A' \\
\emptyset & \text{ otherwise.}
\end{cases}\]
This can be argued directly by expanding to obtain maximal thickets in $\Nbb_{1,2}[T_n]$ representing the elements, and it is easy to see that distinct such semigroups represent distinct elements. As such, we can essentially ignore the content of the uniform subsemigroups and reduce to the corresponding alphabets. Any such is an upward-closed subsets in $\Pcal([n])$, and conversely. As such, if $B_n$ denotes the free Boolean semiring on $n$ generators,
\begin{equation}
|B_n| = \sum_{\substack{U \subseteq \Pcal([n])\\ \text{upward closed}}} 2^{|F|}.	
\end{equation}
These are even more closely related to the Dedekind numbers discussed in Remark \ref{rem:hardcount}, since upward-closed subsets are determined by a choice of antichain! One can also see that these will always have odd cardinality, since the empty subset provides the only odd summand. By manual calculation, we have:
\begin{equation}
|B_n| = 3,7,35,775,\dotsc
\end{equation}
\end{exa}

We leave open the combinatorial investigation of other variants of mirigs to the curious reader.

\bibliographystyle{alpha}
\bibliography{idmonbib}

\newcommand{\etalchar}[1]{$^{#1}$}
\begin{thebibliography}{VHDCG{\etalchar{+}}23}

\bibitem[AA94]{char1}
F.~{Alarc\'{o}n} and D.~D. Anderson.
\newblock Commutative semirings and their lattices of ideals.
\newblock {\em Houston Journal of Mathematics}, 20(4), 1994.

\bibitem[Ber94]{ThueTranslate}
J.~Berstel.
\newblock {Axel} {Thue's} papers on repetitions in words: a translation.
\newblock {\em L.I.T.P., Institut Blaise Pascal, Université Pierre et Marie
  Curie}, 1994.

\bibitem[Boo47]{Boole}
G.~Boole.
\newblock {\em The Mathematical Analysis Of Logic}.
\newblock Macmillan, Barclay, {\&} Macmillan, 1847.

\bibitem[BR83]{squareFree}
J.~Berstel and C.~Reuteneur.
\newblock {\em Square-free words and idempotent semigroups}, volume~17 of {\em
  Encyclopedia of Mathematics and its Applications}.
\newblock Addison-Wesley Publishing Co., 1983.

\bibitem[BV24]{nLabBoolRig}
T.~Bartels and J.B. Vienney.
\newblock {B}oolean rig on the {nLab}.
\newblock \url{https://ncatlab.org/nlab/show/Boolean+rig}, sep 2024.
\newblock \href{https://ncatlab.org/nlab/revision/Boolean+rig/9}{Revision 9}.

\bibitem[Cli54]{Clifford}
A.~H. Clifford.
\newblock Bands of semigroups.
\newblock {\em Proceedings of the American Mathematical Society}, 5(3), June
  1954.

\bibitem[Ega22]{Egan}
G.~Egan.
\newblock \href{https://github.com/nagegerg/IdempotentRig}{IdempotentRig}, in
  {C++}, Dec 2022.
\newblock https://github.com/nagegerg/IdempotentRig.

\bibitem[Eva71]{Evans}
Trevor Evans.
\newblock On {Dedekind's} problem: The number of monotone boolean functions.
\newblock {\em Semigroup Forum}, 2, 1971.

\bibitem[Fra22]{Frankau}
S.~Frankau.
\newblock
  \href{https://github.com/simon-frankau/two-generator-idempotent-rigs}{two-generator-idempotent-rigs},
  in {Rust}, Dec 2022.
\newblock https://github.com/simon-frankau/two-generator-idempotent-rigs.

\bibitem[Ger70]{Gerhard}
J.A. Gerhard.
\newblock The lattice of equational classes of idempotent semigroups.
\newblock {\em Journal of Algebra}, 15, 1970.

\bibitem[GR52]{GreenRees}
J.A. Green and D.~Rees.
\newblock On semigroups in which {$x^r=x$}.
\newblock {\em Mathematical Proceedings of the Cambridge Philosophical
  Society}, 48, 1952.

\bibitem[Gun22]{Gunning}
A.~Gunning.
\newblock \href{https://github.com/agunning/freerig}{freerig}, in {Python}, Dec
  2022.
\newblock https://github.com/agunning/freerig.

\bibitem[Gut09]{Guterman}
A.~E. Guterman.
\newblock {\em Matrix Invariants over Semirings}, volume~6.
\newblock 2009.

\bibitem[{Guz}92]{BSring}
Fernando {Guzm\'{a}n}.
\newblock The variety of {Boolean} semirings.
\newblock {\em Journal of Pure and Applied Algebra}, 78, 1992.

\bibitem[J\"23]{Jakel}
Christian J\"{a}kel.
\newblock A computation of the ninth {Dedekind} number.
\newblock {\em Journal of Computational Algebra}, 614, Sep 2023.

\bibitem[Kis88]{Kisielewicz}
Andrzej Kisielewicz.
\newblock A solution of {Dedekind's} problem on the number of isotone boolean
  functions.
\newblock {\em Journal für die reine und angewandte Mathematik}, 386, 1988.

\bibitem[KM75]{Dedekind}
D.~Kleitman and G.~Markowsky.
\newblock On {Dedekind's} problem: The number of isotone {Boolean} functions.
  {II}.
\newblock {\em Transactions of the American Mathematical Society}, 213, Nov
  1975.

\bibitem[McL54]{McLean}
David McLean.
\newblock Idempotent semigroups.
\newblock {\em The American Mathematical Monthly}, 61(2), February 1954.

\bibitem[SS]{oeis}
N.J.A. Sloane and J.~Shallit.
\newblock Number of elements of a free idempotent monoid on n letters.
\newblock Entry A005345 in the On-Line Encyclopedia of Integer Sequences
  (OEIS), \href{https://oeis.org/A005345}{https://oeis.org/A005345}.

\bibitem[Thu06]{Thue}
A.~Thue.
\newblock {\"{U}ber} unendliche {Zeichenreihen}.
\newblock {\em Norske Vid. Selsk Skr. I Mat.-Nat. Kl.; Christiania}, 1906.

\bibitem[VHDCG{\etalchar{+}}23]{SuperDedekind}
Lennart Van~Hirtum, Patrick De~Causmaecker, Jens Goemaere, Tobias Kenter,
  Heinrich Riebler, Michael Lass, and Christian Plessl.
\newblock A computation of {D(9)} using {FPGA} supercomputing.
\newblock {\em arXiv:2304.03039}, 2023.

\bibitem[Wei79]{char2}
H.J. Weinert.
\newblock A concept of characteristic for semigroups and semirings.
\newblock {\em Acta Sci. Math.}, 41, 1979.

\end{thebibliography}

\end{document}